\pdfoutput=1

\documentclass[pagesize,pdftex]{scrartcl}
\usepackage{latexsym} 
\usepackage[intlimits]{amsmath}
\usepackage{amsthm}
\usepackage{amsfonts}
\usepackage{amssymb}
\usepackage{amsxtra}
\usepackage{amscd}
\usepackage{ifthen}
\usepackage{graphicx}
\usepackage{enumerate}
\usepackage{mathrsfs}
\usepackage{indentfirst}
\usepackage[pagebackref=true]{hyperref}

\hypersetup{
pdfauthor={Georg Sch\"{o}chtel},
pdftitle={Inertial particles in fractional Gaussian fields},
breaklinks=true,
colorlinks=true,
linkcolor=blue,
citecolor=blue,
urlcolor=blue,
filecolor=blue,
}

\numberwithin{equation}{section} 
\setkomafont{title}{\normalfont}

\theoremstyle{plain}
\newtheorem{thm}{Theorem}[section]

\newtheorem{lem}[thm]{Lemma}
\newtheorem{prop}[thm]{Proposition}
\newtheorem{cor}[thm]{Corollary}
\theoremstyle{remark}
\newtheorem{rem}[thm]{Remark}

\newtheorem{assumption}[thm]{Assumption}

\begin{document}

\title{Motion of inertial particles in Gaussian fields driven by an infinite--dimensional fractional Brownian motion}

\author{
Georg Sch\"{o}chtel\thanks{Supported by the DFG and JSPS as a member of the International Research Training Group Darmstadt-Tokyo IRTG 1529.} \\
Technische Universit\"at Darmstadt, Germany \\
Fachbereich Mathematik\\
Email: {\tt schoechtel@mathematik.tu-darmstadt.de}
}

\date{\today}

\maketitle

\begin{abstract}
We study the motion of an inertial particle in a fractional Gaussian random field. The motion of the particle is described by Newton's second law, where the force is proportional to the difference between a background fluid velocity and the particle velocity. The fluid velocity satisfies a linear stochastic partial differential equation driven by an infinite--dimensional fractional Brownian motion with arbitrary Hurst parameter $H \in (0,1)$. The usefulness of such random velocity fields in simulations is that we can create random velocity fields with desired statistical properties, thus generating artificial images of realistic turbulent flows. This model captures also the clustering phenomenon of {\em preferential concentration}, observed in real world and numerical experiments, i.e. particles cluster in regions of low vorticity and high strain rate. We prove almost sure existence and uniqueness of particle paths and give sufficient conditions to rewrite this system as a random dynamical system with a global random pullback attractor. Finally, we visualize the random attractor through a numerical experiment.
\end{abstract}

\noindent\textbf{MSC2010:} 60H15; 35R60; 37H10; 37L30.\\
\noindent\textbf{Keywords:} Gaussian fields; inertial particles; stochastic PDEs; fractional Brownian motion; fractional Ornstein--Uhlenbeck process; random dynamical systems; random attractor.

\section{Introduction}

Motion of inertial (i.e. small heavy) particles in turbulent fluids occurs in natural phenomena as well as in technological processes and therefore has excited theoretical investigations, see for example the book \cite{Zaichik08} and the references therein. Examples for such processes are formation of raindrops, evolution of clouds and combusting of liquid fuel. \\
The starting point for many theoretical investigations concerning the motion of inertial particles in turbulent flows is \emph{Stokes' law} (see e.g. \cite{Batchelor67}), which says that the force exerted by the fluid on the particle is proportional to the difference between the background fluid velocity and the particle velocity, i.e. we are concerned with the following transport equation
\begin{equation}
\tau \ddot{x}(t) = v \big(x(t),t \big) - \dot{x}(t),
\label{Sec-1-eq1-Stokes-law}
\end{equation}
where $v(t,x)$ is the velocity of the fluid at point $x$ in space at time $t$ and $x(t)$ is the position of the particle at time $t$. Here the \emph{response time} $\tau = \frac{m}{\nu C}$ in two dimensions or $\tau= \frac{m}{6 \pi \nu r}$ in three dimensions is often called \emph{Stokes' time}, where $m$ is the particles mass, $r$ the particles radius, $\nu$ the fluid \emph{viscosity} and $C>0$ a universal constant. We neglect that $C$ actually depends on the radius $r$ and the relative velocity, making the law non-linear \cite{Batchelor67}. Further, an important non--dimensional parameter related to the equation (\ref{Sec-1-eq1-Stokes-law}) is the so--called \emph{Stokes' number} $St$, which is the ratio of the particle aerodynamic time constant to an \emph{appropriate} turbulence time scale. In turbulent fluid flows $St$ is usually defined by $St=\tau / \tau_\eta$, where $\tau_\eta= \overline{\epsilon}^{-1/2}\nu^{1/2}$ is the \emph{eddy turnover time} associated to the \emph{Kolmogorov length scale} $\eta=\overline{\epsilon}^{-1/4}\nu^{3/4}$ with viscosity $\nu$ and \emph{mean energy dissipation rate} $\overline{\epsilon}$. If $\tau \to \infty$ the equation (\ref{Sec-1-eq1-Stokes-law}) tends to $\ddot{x}(t) =0$, which is the equation of motion a of particle moving with constant velocity. And if $\tau=0$, i.e. the inertia of the particle is neglected, we get $ \dot{x}(t)= v \big( x(t),t \big) $, which is the equation of motion of a \emph{fluid particle} or a \emph{passive tracer}, which is a particle that follows the streamlines of the fluid. Various extensions of the basic model (\ref{Sec-1-eq1-Stokes-law}) have been considered in the literature, in particular by Maxey and collaborators \cite{Maxey87-1, Maxey87-2, Wang92}. \\
Real world and numerical experiments show that the distribution of inertial particles in a turbulent fluid is highly correlated with the turbulent motion (\cite{Eaton94, Fessler94, Reade00, Squires91-1, Squires91-2}). The particles cluster in regions of low vorticity and high strain rate. This clustering phenomenon is known as \emph{preferential concentration}. The Stokes' number $St$ plays a central role in the effect of preferential concentration. If the particle and fluid time constant have almost the same order, i.e. $St \approx 1$, the particles concentrate in regions where straining dominates vorticity. Experiments at high or low Stokes' numbers do not show this clustering phenomenon. \\
A model for the motion of inertial particles in two dimensions which covers the preferential concentration phenomenon was introduced by Sigurgeirsson and Stuart in \cite{Sigurgeirsson02-1} and analyzed in a series of papers \cite{Hairer04, Kupferman04, Pavliotis03, Pavliotis05-1, Pavliotis05-2, Pavliotis06, Pavliotis10, Sigurgeirsson02-2}. This model consists of Stokes' law (\ref{Sec-1-eq1-Stokes-law}), where the velocity field is a Gaussian random field that is incompressible, homogeneous, isotropic and periodic in space, and stationary and Markovian in time. This gives the equations in \emph{non--dimensional} form
\begin{align}
 	\tau \ddot{x}(t) &= v \left( x(t),t \right) - \dot{x}(t), \quad \left( x(0),\dot{x}(0) \right) \in \mathbb{T}^2 \times \mathbb{R}^2, \tag{$S1$} \\
 	v(x,t) &= \nabla^{\bot} \psi(x,t)= \Big( \frac{\partial \psi}{\partial x_2} (x,t), - \frac{\partial \psi}{\partial x_1} (x,t) \Big), \tag{$S2$} \\
 	d \psi_t &= \gamma A \psi_t dt + \gamma^{\frac{1}{2}}  Q^{\frac{1}{2}} dW_t, \quad \psi_0 \in V, \text{ } t \geq 0, \tag{$S3$}
\end{align}
where $\tau, \gamma >0$, $\mathbb{T}^2$ is the two--dimensional torus, $W$ is an infinite--dimensional Brownian motion with an appropriate infinite--dimensional separable Hilbert space $V$ as state space and the self--adjoint operators $A$ and $Q^{\frac{1}{2}}$ on $V$ will be chosen to match a desired energy spectrum of the velocity field. The equation ($S3$) is interpreted as a linear stochastic evolution equation (\cite{DaPrato-Zabczyk-92}). Precise assumptions will be given later. In \cite{Sigurgeirsson02-1} various qualitative properties of the system ($S$) have been studied, such as existence and uniqueness of solutions and existence of a random attractor. In particular, numerical simulation in \cite{Sigurgeirsson02-1}, see also \cite{Bec03}, indicates that system ($S$) also covers the preferential concentration phenomenon, where now the parameter $\tau$ can be interpreted as Stokes' number, whereas $\gamma$ indicates how fast the velocity field decorrelates. Therefore, the random attractor of the system ($S$) is highly relevant for the study of preferential concentration of inertial particles. \\
Further, various limits of physical interest in system ($S$) have been studied: rapid decorrelation in time limits \cite{Kupferman04, Pavliotis03, Pavliotis05-1} and diffusive scaling limits (homogenization) \cite{Hairer04, Pavliotis05-2, Pavliotis06, Pavliotis10}. \\
Also of interest are studies of system ($S$) for fluid particles or passive tracers, i.e. $\tau = 0$ and so ($S1$) is replaced by $\dot{x}(t)= v \left( x(t),t \right)$. This problem in a similar framework was considered among others by Carmona, see \cite{Carmona98} and references therein. \\
In this work we generalize the results of Sigurgeirsson and Stuart in \cite{Sigurgeirsson02-1} to the fractional noise case, i.e. in the system ($S$) the stochastic evolution equation ($S3$) will be driven by an infinite--dimensional fractional Brownian motion $B^H$ with arbitrary Hurst parameter $H \in (0,1)$. We recover system ($S$) as a special case for $H= \frac{1}{2}$. In particular, we improve Sigurgeirsson and Stuart's assertion in \cite{Sigurgeirsson02-1} concerning the existence and uniqueness of the random attractor in such a way that we extend the universe of attracting sets from deterministic bounded sets to random \emph{tempered} sets with an explicit representation of a random tempered the universe absorbing set. The main motivation to use fractional noise, to model now a fractional Gaussian random velocity field, is that there exist certain statistical similarities in the scaling behaviour of a fractional Brownian motion and a turbulent velocity field. These statistical similarities are described in the next section. \\
Also of interest and what Sigurgeirsson and Stuart have not done, is to give analytic (upper) bounds of the Hausdorff dimension of the random attractor, which is almost sure constant due to the ergodicity of the noise. This will be studied also in the fractional noise case in a forthcoming paper. \\
The remainder of this article is organized as follows: In section 2 we present the motivation to use fractional noise to model the random velocity field. Section 3 provides properties of a stationary fractional Ornstein--Uhlenbeck process, which are needed in Section 4, where we introduce the generalized model ($S$) for motion of inertial particles in a fractional Gaussian random velocity field and prove almost sure existence and uniqueness of particle paths. Section 5 describes how to match desired statistical properties, in particular a favoured energy spectrum of the velocity field. Further, we verify in Section 6 that the system ($S$) defines a random dynamical system, which admits a random pullback attractor. Finally, we visualize the random attractor through a numerical experiment in Section 7. \\
In the whole article we use the following conventions: If $E$ is a Banach space, then we denote by $\vert \cdot \vert_E$ the norm of $E$ and by $ \langle \cdot, \cdot \rangle_E (=: \vert \cdot \vert^2_E)$ the inner product of $E$, if $E$ is even a Hilbert space. In the cases $E= \mathbb{R}^n$ for some $n \in \mathbb{N}$ or $E= \mathbb{T}^2 \times \mathbb{R}^2$, where $\mathbb{T}^2 $ denotes the two--dimensional torus, we just write $\vert \cdot \vert = \vert \cdot \vert_E$ and $\langle \cdot, \cdot \rangle = \langle \cdot, \cdot \rangle_E$ for the norm and inner product, respectively. We also use $\vert z \vert$ to denote the absolute value of a scalar $z \in \mathbb{R}$ or a complex number $ z \in \mathbb{C}$. Further, by $*$ we denote the complex conjugate of a complex number or a complex--valued function. For a separable Hilbert space $E$ we denote by $\mathcal{L}(E)$ the Banach space of linear, bounded operators form $E$ into $E$ equipped with the operator norm $\vert \cdot \vert_{\mathcal{L}(E)}$ and $\mathcal{L}_2(E)$ equipped with the inner product $\langle T, S \rangle_{\mathcal{L}_2(E)} := \sum_{n \in \mathbb{N}} \langle Te_n, Se_n \rangle_E$, $T,S \in \mathcal{L}_2(E)$, denotes the Hilbert space of Hilbert--Schmidt operators from $E$ into $E$, where $(e_n)_{n \in \mathbb{N}}$ is a orthonormal basis of $E$. Further, we denote by $\left( \Omega, \mathcal{F}, \mathbb{P} \right)$ always a probability space and $\mathbb{E}$ the expectation w.r.t. $\mathbb{P}$. For $0 < p < \infty$ and $E$ a separable Banach space, we write $L^p(\Omega, E)$ and $L^p(\Omega):= L^p(\Omega, \mathbb{R})$ if $E = \mathbb{R}$ for the Banach space of $E$--valued integrable random variables $X: \Omega \rightarrow E$ (in fact, equivalence classes of random variables, where $X \sim Y$ if $X=Y$ $\mathbb{P}$--a.s.) equipped with the norm $ \vert X \vert_{L^p}^p = \mathbb{E} (\vert X \vert^p_E)$. In the same way we define the spaces $L^p(\mathbb{R})$ with respect to the Lebesgue measure. We denote by $\mathcal{B}(E)$ the Borel $\sigma$--algebra of a metric space $E$. Finally, for a multi--index $\delta=(\delta_1, \dots, \delta_N) \in \mathbb{N}^N_0$ we set $\vert \delta \vert:= \delta_1 + \dots + \delta_N$ and denote the partial derivative operator by $ D^\delta:= \partial^{\vert \delta \vert} / \partial x_1^{\delta_1} \dots \partial x_N^{\delta_N}$ and $D^\delta:= id$ if $\vert \delta \vert=0$.


\section{Motivation for the use of fractional noise}

Since solutions of the Navier--Stokes equations for very turbulent fluids, i.e. at large \emph{Reynolds numbers}, are unstable in view of the sensitive dependence on the initial conditions that makes the fluid flow irregular both in space and time, a \emph{statistical description} is needed (see e.g. \cite{Chorin}). Based on this, assume that $v(x,t)$, $x \in \mathbb{R}^2$, $t \geq 0$, is a measurable, time \emph{stationary}, space \emph{homegeneous} and \emph{(local) isotropic} random field on a probability space $\left( \Omega, \mathcal{F}, \mathbb{P} \right)$, such that $v$ satisfies $\omega$--wise, $\omega \in \Omega$, the Navier--Stokes equations in two dimensions. Due to a \emph{phemenological approach}, first introduced by Kolmogorov (\cite{Kolmogorov41}) in three dimensions and by Kraichnan (\cite{Kraichnan67}), Leith (\cite{Leith68}) and Batchelor (\cite{Batchelor69}) in the two dimensional case, we have the following phenomenological correspondence (see e.g. \cite{Babiano85}) between the relation of the \emph{second--order
structure function} 
\begin{equation}
\mathbb{E} \left( \vert v(x+r,t) - v(x,t) \vert^2 \right) = C \vert r \vert^{\alpha - 1}
\label{Sec2-eq1}
\end{equation}
and the relation of the \emph{energy spectrum} $E(\cdot)$ of the velocity field $v$
\begin{equation}
E \left( k  \right)  = \widetilde{C} k^{- \alpha},
\label{Sec2-eq2}
\end{equation}
where $C, \widetilde{C} > 0$ are some constants, $1 < \alpha < 3$, $r \in \mathbb{R}^2$ and $k > 0$ in the \emph{inertial subrange}. In particular, for $\alpha = 5 / 3$ we obtain the famous \emph{Kolmogorov's two--thirds law} and \emph{Kolmogorov's five--thirds law} (or \emph{Kolmogorov energy spectrum}), respectively. \\
The connection to the fractional Brownian motion gives now \emph{Taylor's frozen turbulence hypothesis} (\cite{Taylor38}) which informally assumes that the spatial pattern of turbulent motion is unchanged as it is advected by a constant (in space and time) mean velocity $\overline{V}$, $ \vert \overline{V} \vert := \left( \sum_i \overline{V}_i^2 \right)^{\frac{1}{2}}$, let us say along the $\overline{x}$ axis. Mathematically, Taylor's hypothesis says that for any \emph{scalar--valued fluid--mechanics variable} $\xi$ (e.g. $v_i$, $i=1,2$) we have
\begin{equation}
\frac{\partial \xi}{\partial t}= - \vert \overline{V} \vert \frac{\partial \xi}{\partial \overline{x}}.
\label{Sec2-eq3}
\end{equation}
The frozen turbulence hypothesis enables us to express the statistical characteristics of the space differences $v(x+r,t)-v(x,t)$ in terms of the time differences $v(x,t)- v(x,t + s)$ corresponding to a fixed time $t$. Indeed, (\ref{Sec2-eq3}) implies
$v_i(x, t + s )= v_i(x- \overline{V} s, t)$ and therefore by (\ref{Sec2-eq1}) we deduce
\begin{equation}
\mathbb{E} \left( \vert v(x,t)-v(x,t + s) \vert^2 \right) = C \vert \overline{V}\vert^{\alpha -1} s^{\alpha - 1}
\label{Sec2-eq4}
\end{equation}
in the inertial subrange along the time axis. For a discussion under which conditions the frozen turbulence hypothesis is valid, see \cite{Monin75}. \\
Note that due to our derivation the properties (\ref{Sec2-eq2}) and (\ref{Sec2-eq4}) are closely related to each other! \\
Now comparing (\ref{Sec2-eq4}) with the statistical property $\mathbb{E} \left( \vert \beta^H_{t+ s} - \beta^H_t \vert^2 \right) = s^{2H}$ of the fractional Brownian motion (fBm) $(\beta_t^H)_{t \in \mathbb{R}}$ with Hurst parameter $H \in (0,1)$
indicates that it is reasonable to model the random velocity field $v$ with noise driven by a fBm. With this motivation Shao \cite{Shao95}, Sreenivasan et al. \cite{Juneja94} and Papanicolaou et al. \cite{Papanicolaou03} proposed models of the velocity field driven by a finite--dimensional fBm, always with special interest in the case $H= 1/3$ in view of the Kolmogorov energy spectrum. \\
To capture both statistical features, (\ref{Sec2-eq2}) and (\ref{Sec2-eq4}), we introduce in Section 4 a two--dimensional, incompressible, stationary, homogeneous and isotropic random velocity field. These assumptions (Gaussian statistics, stationarity and isotropy) are quite common in random mathematical models of turbulent fluids, see e.g. \cite{Batchelor67}.


\section{Stationary fractional Ornstein-Uhlenbeck process}

As will be seen in the next section, the stationary solution of our stochastic evolution equation will be given by an infinite series of stationary fractional Ornstein--Uhlenbeck processes. For that purpose we recall in this section from \cite{Cheridito03} basic properties of the stationary fractional Ornstein--Uhlenbeck process. \\
First, we remind that a real--valued and normalized fractional Brownian motion (fBm) on $\mathbb{R}$ with Hurst parameter $H \in (0,1)$ is a Gaussian process $\beta^H = (\beta^H_t)_{t \in \mathbb{R}}$ on a probability space $\left( \Omega, \mathcal{F}, \mathbb{P} \right)$, having the properties $\beta^H_0 =0$ $\mathbb{P}$-a.s., $\mathbb{E} \left(\beta^H_t \right)=0$, $\mathbb{E} \left(\beta^H_t\beta^H_s \right)= \frac{1}{2} \left( \vert t \vert^{2H} + \vert s \vert^{2H} - \vert t-s \vert^{2H} \right), \ s,t \in \mathbb{R}$, with continuous sample paths $\mathbb{P}$-a.s.. The reader, interested in stochastic calculus w.r.t. fBm, is referred to \cite{Mishura08}. \\  
We consider now the fractional Langevin equation
\begin{equation}
X_t= X_0 - \nu \alpha \int_0^t X_s ds + \nu^H \sqrt{\lambda} \beta_t^H, \quad t \geq 0, \ X_0 \in \mathbb{R},
\label{Sec3-eq1}
\end{equation}
where $\alpha, \lambda, \nu > 0$. The unique stationary solution of (\ref{Sec3-eq1}) is given by the fractional Ornstein--Uhlenbeck process $Y_t:= \nu^H \sqrt{\lambda}  \int_{- \infty}^t e^{- (t-u) \nu \alpha } d\beta_u^H$, $t \in \mathbb{R}$, see \cite{Cheridito03}. The uniqueness has to be understood as uniqueness in law in the class of stationary solutions adapted to the natural filtration generated by the two--sided fBm $\beta^H$. \\
The following proposition will be used several times in this article. Here and in the following $\Gamma(x)=\int_0^{\infty}s^{x-1} e^{-s}ds$, $x > 0$, denotes the gamma function.
\begin{prop}
Let $Y_t$, $t \in \mathbb{R}$, be the unique stationary solution to (\ref{Sec3-eq1}) and set $C(H):= \Gamma(2H+1) \sin(\pi H) / \pi >0$.
\begin{enumerate}
	\item[(i)] Then for all $t, s \in \mathbb{R}$ we have 
\begin{equation}
Cov \left(Y_t , Y_s \right) := \mathbb{E} \left( Y_t Y_s \right) = C(H) \frac{\lambda}{\alpha^{2H}} \int_0^\infty \cos((t-s) \nu \alpha x) \frac{ x^{1-2H}}{ 1 + x^2} dx. 
\label{Sec3-eq2}
\end{equation}
In particular, for all $t \in \mathbb{R}$ we have $Var(Y_t):= \mathbb{E} \left( Y_t^2 \right) = \lambda \Gamma(2H)H / \alpha^{2H}.$ \\
	\item[(ii)] Fix $T>0$. Then there is a constant $C_1(H, \lambda, \nu, \alpha, T) >0$ such that for any $t,s \in [-T,T]$ we have
	$$ C_1(H, \lambda, \nu, \alpha, T) \vert t-s \vert^{2H} \leq \mathbb{E} \left ( \vert Y_t - Y_s \vert^2 \right). $$
	\item[(iii)] Then for any $\gamma \in (0,H]$ there is a constant $C_2(H, \nu, \gamma) >0$ such that for any $t,s \in \mathbb{R}$ we have
\begin{equation}
\mathbb{E} \left ( \vert Y_t - Y_s \vert^2 \right) \leq C_2(H, \nu, \gamma) \lambda \alpha^{2 \gamma - 2H} \vert t - s \vert^{2\gamma}.
\label{Sec3-eq3}
\end{equation}
	\end{enumerate}
\label{Section3-Prop1}
\end{prop}
\begin{proof}
(i): By Remark 2.4 in \cite{Cheridito03} we have $$Cov \left(Y_t, Y_s \right) = \nu^{2H} \lambda \frac{\Gamma(2H+1) \sin(\pi H)}{2 \pi} \int_{- \infty}^{\infty} e^{i (t-s) x} \frac{ \vert x \vert^{1-2H}}{( \nu \alpha)^2 + x^2} dx,$$ where $t,s \in \mathbb{R}$ and $i$ denotes the imaginary unit. (\ref{Sec3-eq2}) follows now by the change of variables $y= x / (\nu \alpha)$ and by the fact that $\sin(\cdot)$ and $\cos(\cdot)$ are odd and even functions, respectively. The relation for the variance of $Y$ is the consequence of (\ref{Sec3-eq2}) with the relations $\int_{0}^{\infty} \frac{  x^{1-2H}}{1 + x^2} dx = \frac{\Gamma(1-H) \Gamma(H)}{2 \Gamma(1)} $, see (3.241.2) in \cite{Gradshteyn00}, $\Gamma(1-H) \Gamma(H) = \frac{\pi}{\sin(\pi H)}$ and $\Gamma(2H+1)= \Gamma(2H)2H$. \\
(ii)+(iii): Set $C(H):= \frac{\Gamma(2H+1) \sin(\pi H)}{ \pi}$ and fix without loss of generality $- \infty < s < t$. By $(i)$ and the change of variables $z= (t-s) \nu \alpha x$ we have
\begin{equation}
\begin{split}
\mathbb{E} \left ( \vert Y_t - Y_s \vert^2 \right) &= 2C(H) \frac{\lambda}{\alpha^{2H}} \int_0^\infty (1-\cos((t-s) \nu \alpha x)) \frac{ x^{1-2H}}{ 1 + x^2} dx \\
		&= 2C(H)\lambda (\nu (t-s))^{2H} \int_0^\infty (1-\cos(z)) \frac{ z^{1-2H}}{ ((t-s)\nu \alpha)^2 + z^2} dz.
\end{split}
\label{Sec3-eq4}
\end{equation}
Lower bound: For any $-T \leq s < t \leq T$, where $T> 0$ is fixed, we obtain by (\ref{Sec3-eq4})
\begin{equation*}
\begin{split}
\mathbb{E} \left ( \vert Y_t - Y_s \vert^2 \right) &\geq 2C(H)\lambda (\nu (t-s))^{2H} \inf_{-T \leq s< t \leq T} \int_0^\infty (1-\cos(z)) \frac{ z^{1-2H}}{ ((t-s)\nu \alpha)^{2} + z^2} dz \\
		&= 2C(H)\lambda (\nu (t-s))^{2H} \int_0^\infty (1-\cos(z)) \frac{ z^{1-2H}}{ (2T \nu \alpha)^2 + z^2} dz.
\end{split}
\end{equation*}
Upper bound: By (\ref{Sec3-eq4}) we have for all $- \infty < s < t$
\begin{equation*}
\begin{split}
\mathbb{E} \left ( \vert Y_t - Y_s \vert^2 \right) &\leq 2C(H)\lambda (\nu (t-s))^{2H} \sup_{- \infty < s < t} \int_0^\infty (1-\cos(z)) \frac{ z^{1-2H}}{ ((t-s)\nu \alpha)^2 + z^2} dz \\
		&= 2C(H)\lambda (\nu (t-s))^{2H} \int_0^\infty (1-\cos(z)) z^{-1-2H} dz.
\end{split}
\end{equation*}
Notice that $\int_0^\infty (1-\cos(z)) z^{-1-2H} dz = 2 \int_0^\infty \sin^2(z/2) z^{-1-2H} dz$ and this indefinite integral is finite. Indeed, this follows from the fact that we have $0<H<1$ and $\sin^2(z) = \vert \sin(z) - \sin(0) \vert^2 \leq C(\epsilon) z^{2 \epsilon}$ for any $\epsilon \in[0,1]$ and a constant $C(\epsilon)>0$. \\
It is left to prove the relation (\ref{Sec3-eq3}) for $\gamma \in (0,H)$. By (\ref{Sec3-eq2}) and using again the H\"older continuity of $\sin(\cdot)$, we get
\begin{equation}
\begin{split}
\mathbb{E} \left ( \vert Y_t - Y_s \vert^2 \right) &= 2C(H) \frac{\lambda}{\alpha^{2H}} \int_0^\infty (1-\cos((t-s) \nu \alpha x)) \frac{ x^{1-2H}}{ 1 + x^2} dx \\
		&= 4C(H) \frac{\lambda}{\alpha^{2H}} \int_0^\infty \sin^2((t-s) \nu \alpha x/2) \frac{ x^{1-2H}}{ 1 + x^2} dx  \\
		&\leq \widetilde{C}(H, \gamma) \lambda \alpha^{2\gamma - 2H} (\nu (t-s))^{2 \gamma} \int_0^\infty \frac{ x^{1+2\gamma -2H}}{ 1 + x^2} dx < \infty, 
\end{split}
\label{Sec3-eq5}
\end{equation}
for any $\gamma \in (0, H)$ with a constant $\widetilde{C}(H, \gamma)>0$. 
\end{proof}


\section{The model}

The concrete model we study is the following system ($M$) of equations in non--dimensional form
\begin{align}
 	\tau \ddot{x}(t) &= v \big( x(t),t \big) - \dot{x}(t), \quad \big( x(0),\dot{x}(0) \big) \in \mathbb{T}^2 \times \mathbb{R}^2, \tag{$M1$} \\
 	v(x,t) &= \nabla^{\bot} \psi(x,t)= \left( \frac{\partial \psi}{\partial x_2} (x,t), - \frac{\partial \psi}{\partial x_1} (x,t) \right), \tag{$M2$} \\
 	d \psi_t &= \nu A \psi_t dt + \nu^H  Q^{\frac{1}{2}} dB^H_t, \quad \psi_0 \in V, \text{ } t \geq 0, \tag{$M3$}
\end{align}
where we assume that
\begin{assumption}
\begin{enumerate}
	\item[(i)] $\tau, \nu > 0$.
	\item[(ii)] $V:= \{ f \in L^{2,per}(\mathbb{T}^2) | \ \int_{\mathbb{T}^2} f(x) dx = 0 \}$ is the separable Hilbert space with inner product $\langle f, g \rangle_V:= \int_{\mathbb{T}^2} f(x)g^*(x)dx$, $f,g \in V$, and with orthonormal basis (ONB) $ \big( e_k(\cdot) \big)_{k \in K}:= \big( e^{i \langle k, \cdot \rangle} \big)_{k \in K}$, where $k \in K:= 2 \pi \mathbb{Z}^2 \backslash \{ (0,0) \}$ and $i$ denotes here and in the following the imaginary unit.
	\item[(iii)] $A: \mathcal{D}(A) \subset V \rightarrow V$ is a linear operator such that there is a strictly positive sequence $\big( \alpha_k \big)_{k \in K} \subset [c, \infty)$ with $c>0$, $\alpha_k = \alpha_{-k}$, $ Ae_k = - \alpha_k e_k $ and $ \alpha_k \to \infty$ for $\vert k \vert \to \infty$.
	\item[(iv)] $Q^{\frac{1}{2}}: V \rightarrow V$ is a bounded linear operator such that there is a positive sequence $\big( \sqrt{\lambda_k} \big)_{k \in K} \subset [0, \infty)$ with  $ \sqrt{\lambda_k} = \sqrt{\lambda_{-k}}$ and $ Q^{\frac{1}{2}}e_k = \sqrt{\lambda_k} e_k $.
	\item[(v)] $ \big( B_t^H \big)_{t \geq 0}$ is an infinite--dimensional fractional Brownian motion in $V$ with Hurst parameter $H \in (0,1)$ defined on a probability space $(\Omega , \mathcal{F}, \mathbb{P})$ by the formal series
\begin{equation}
B^H(t)= \sum_{k \in K} \beta_k^H(t)e_k, \quad t \in \mathbb{R}, 
\label{Sec4-eq1}
\end{equation}
	where $\big( (\beta_k^H(t))_{t \in \mathbb{R}}, k \in K \big)$ is a sequence of complex--valued and normalized fractional Brownian motions, each with the same fixed Hurst parameter $H \in (0,1)$, i.e. $\beta_k^H = \frac{1}{ \sqrt{2}} Re(\beta_k^H)+ i \frac{1}{ \sqrt{2}} Im(\beta_k^H)$, where $Re(\beta_k^H)$ and $Im(\beta_k^H)$ are independent real--valued and normalized fractional Brownian motions on $\mathbb{R}$, and different $\beta_k^H$ are independent except $\beta_{-k}^H=(\beta_k^H)^*$.
\end{enumerate}
\label{Main-Assumption}
\end{assumption}
So by Assumption \ref{Main-Assumption} $A$ is a strictly negative, self--adjoint operator with compact resolvent and $Q^{\frac{1}{2}}$ a positive, self--adjoint operator, respectively. We will show in this section that the conditions on the sequences $(\beta_k^H)_{k \in K}$, $(\lambda_k)_{k \in K}$ and $(\alpha_k)_{k \in K}$ in Assumption \ref{Main-Assumption}, together with some growth conditions on $(\lambda_k)_{k \in K}$ and $(\alpha_k)_{k \in K}$, imply that $\psi(x,t)$ and the components of $v(x,t)=\nabla^{\bot} \psi(x,t)$ are real--valued. \\
Under the Assumption \ref{Main-Assumption} on the operator $A$, it follows that $\nu A$ generates an \emph{analytic semigroup} on $V$ which in the following will be denoted by $(S_t)_{t \geq 0}$. Since the spectrum of $\nu A$ is strictly negative and has a lower bound strictly less than zero, the semigroup $(S_t)_{t \geq 0}$ is \emph{exponentially stable}. In particular, we have $ \vert S(t) \vert_{\mathcal{L}(V)} \leq e^{-t \nu \inf_{k \in K} \alpha_k} \leq e^{-t c \nu} $ for all $t \geq 0$. \\
The domain $\mathcal{D}(A)$ of $A$ is given by
$ \mathcal{D}(A):= \{ f \in V | \ \sum_{k \in K} \alpha_k^2 \vert \langle f, e_k \rangle_V \vert^2 < \infty \}.$
Further, we define fractional powers $(-A)^{\gamma}: \mathcal{D}((-A)^{\gamma}) \subset V \rightarrow V$, $\gamma \geq 0$, of the strictly positive operator $(-A)$ by
$ \mathcal{D}((-A)^{\gamma}):= \{ f \in V | \ \sum_{k \in K} \alpha_k^{2 \gamma} \vert \langle f, e_k \rangle_V \vert^2 < \infty \}.$ 
$\mathcal{D}((-A)^{\gamma})$ endowed with the inner product 
$$\langle (-A)^{\gamma} f, (-A)^{\gamma} g \rangle_V = \sum_{k \in K} \alpha_k^{2 \gamma} \langle f, e_k \rangle_V \langle g, e_k \rangle_V^* =: \langle f, g \rangle_{(-A)^{\gamma}}  $$
for $f,g \in \mathcal{D}((-A)^{\gamma})$, becomes a Hilbert space. Notice also that $\mathcal{D}( \nu A) = \mathcal{D}(A)$ and $\mathcal{D}((- \nu A)^{\gamma}) = \mathcal{D}((-A)^{\gamma})$ for any $\nu > 0$ and $\gamma \geq 0$. We similarly define $\mathcal{D}((-A)^{\gamma})$ for $\gamma < 0$ as the completion of $V$ for the norm $\vert \cdot \vert_{(-A)^{\gamma}}$. \\
We are mainly interested in the special case when $A= \Delta$, where $\Delta$ denotes the Laplace operator with periodic boundary conditions. Then $\alpha_k = \vert k \vert^2$, $k \in K$, and $ \mathcal{D}((-\Delta))= W^{2,2}(\mathbb{T}^2) \cap V$, where $W^{2,2}(\mathbb{T}^2) \cap V$ denotes the Sobolev space of functions on $\mathbb{T}^2$ whose weak derivatives up to order 2 are in $V$. In particular, we have $ \mathcal{D}((-\Delta)^{\gamma})= W^{2 \gamma,2}(\mathbb{T}^2) \cap V$ for $\gamma \geq 0$. \\
We will only use the following concept of solutions to equation ($M3$). For that fix $T>0$. A $\mathcal{B}([0,T]) \otimes \mathcal{F}$--measurable $V$--valued process $(\psi(t))_{t \in [0,T]}$ is said to be a \emph{mild solution} of ($M3$), if for all $t \in [0,T]$ 
\begin{equation}
\psi(t)= S(t)\psi(0) + \int_0^t S(t-s) Q^{\frac{1}{2}} dB^H(s)
\label{Sec4-eq2}
\end{equation}
$\mathbb{P}$-a.s., where the stochastic integral on the right hand side of (\ref{Sec4-eq2}) is defined by
\begin{equation}
\int_0^t S(t-s) Q^{\frac{1}{2}} dB^H(s):= \sum_{k \in K} \sqrt{\lambda}_k \nu^{H} \int_0^t e^{-(t-s) \nu \alpha_k} d \beta^H_k(s) \ e_k,
\label{Sec4-eq3}
\end{equation}
provided the infinite series in (\ref{Sec4-eq3}) converges in $L^2(\Omega, V)$. \\
We are mainly interested in \emph{strictly stationary solutions} of ($M3$) which are \emph{ergodic}. Therefore, we call a mild solution $(\psi(t)_{t \geq 0})$ \emph{strictly stationary}, if for all $k \in \mathbb{N}$ and for all arbitrary positive numbers $t_1, t_2, \dots, t_k$, the probability distribution of the $V^k$--valued random variable $\big( \psi(t_1+r), \psi(t_2+r), \dots, \psi(t_k+r) \big)$ does not depend on $r \geq 0$, i.e.
$$ Law \left( \psi(t_1+r), \psi(t_2+r), \dots, \psi(t_k+r) \right) = Law \left( \psi(t_1), \psi(t_2), \dots, \psi(t_k) \right) $$
for all $t_1, t_2, \dots, t_k, r \geq 0$. Here $Law(\cdot)$ denotes the probability distribution. We say that a strictly stationary mild solution $(\psi(t))_{t \geq 0}$ of ($M3$) is \emph{unique} if every strictly stationary mild solution of ($M3$) which is adapted to the natural filtration generated by the two--sided infinite--dimensional fractional Brownian motion (\ref{Sec4-eq1}) has the same distribution as $(\psi(t))_{t \geq 0}$. Further, we call a strictly stationary solution \emph{ergodic} if for all measurable functionals $\rho : V \to \mathbb{R}$ such that $\mathbb{E} \left(\vert \rho(\psi(0)) \vert \right) < \infty$ we have $\mathbb{P}$-a.s. 
$$ \lim_{T \to \infty} \frac{1}{T} \int_0^T \rho(\psi(t)) dt = \mathbb{E}(\rho(\psi(0))).$$
We have the following existence and uniqueness result.
\begin{thm}
Suppose Assumption \ref{Main-Assumption} holds and assume that there is $\epsilon > 0$ such that 
\begin{equation}
\sum_{k \in K} \lambda_k \alpha_k^{2(\epsilon- H)} < \infty.
\label{Sec4-eq3-2}
\end{equation}
Then there exists a unique ergodic mild solution $ \psi$ to equation ($M3$) given by
\begin{equation}
\psi(t)= \sum_{k \in K} \sqrt{\lambda_k} \nu^H \int_{- \infty}^t e^{-(t-u) \nu \alpha_k} d\beta_k^H(u) e_k, \ t \in \mathbb{R}.
\label{Sec4-eq4}
\end{equation}
\label{Sec4-thm1}
\end{thm}
\begin{proof}
The existence of a strictly stationary mild solution which is ergodic follows by Theorem 3.1 and Theorem 4.6 in \cite{Maslowski08}. We only remark that Theorem 3.1 and Theorem 4.6 in \cite{Maslowski08} are applicable since we have in particular
\begin{equation*}
\begin{split}
\vert S(t)Q^{\frac{1}{2}} &\vert_{\mathcal{L}^2(V)} =  \vert (-\nu A)^\gamma (-\nu A)^{- \gamma} S(t)Q^{\frac{1}{2}} \vert_{\mathcal{L}^2(V)}  = \vert (-\nu A)^\gamma S(t)(-\nu A)^{- \gamma} Q^{\frac{1}{2}} \vert_{\mathcal{L}^2(V)} \\
		  &\leq \vert (-\nu A)^\gamma S(t) \vert_{\mathcal{L}(V)}  \vert (-\nu A)^{- \gamma} Q^{\frac{1}{2}} \vert_{\mathcal{L}^2(V)} = \vert (-\nu A)^\gamma S(t) \vert_{\mathcal{L}(V)} \left( \sum_{k \in K} \frac{\lambda_k}{(\nu \alpha_k)^{2 \gamma}} \right)^{\frac{1}{2}} \\
		  &\leq C t^{- \gamma} \left( \sum_{k \in K}  \frac{\lambda_k}{(\nu \alpha_k)^{2 \gamma}} \right)^{\frac{1}{2}} < \infty
\end{split}
\end{equation*}
for any $\gamma \in [ \max \{0,H - \epsilon \}, \infty)$, a constant $C> 0$ and any $t > 0$. Here we used (\ref{Sec4-eq3-2}) and the well--known interpolation inequality $\vert (-\nu A)^\gamma S(t) \vert_{\mathcal{L}(V)} \leq C t^{- \gamma}$, see e.g. Theorem 2.6.3 in \cite{Pazy83}. \\
Now assume that we have two strictly stationary mild solutions $\psi$ and $\widetilde{\psi}$ to equation ($M3$). Notice that
$$ \vert \psi(t)- \widetilde{\psi}(t) \vert_V= \vert S(t)(\psi(0) - \widetilde{\psi}(0)) \vert_V \leq e^{-t \nu \inf_{k \in K} \alpha_k} \vert \psi(0) - \widetilde{\psi}(0) \vert_V \to 0 \quad \text{ for } t \to \infty $$
$\mathbb{P}$-a.s. and this implies uniqueness in the sense of our definition. The representation (\ref{Sec4-eq4}) of the mild solution is just the consequence of the definition of a mild solution and stationarity.
\end{proof}
The next remark will be useful in this and the following sections.
\begin{rem}
Let $\psi$ be the unique ergodic mild solution to equation ($M3$) given by (\ref{Sec4-eq4}). For any $t \in \mathbb{R}$ and $k \in K$ we set in the following
\begin{equation*}
\begin{split}
\hat{\psi}_k (t) &:= \sqrt{\lambda_k} \nu^H \int_{-\infty}^t e^{-(t-u) \nu \alpha_k} d\beta_k^H(u) \\
	&= \sqrt{\frac{\lambda_k}{2}} \nu^H \int_{-\infty}^t e^{-(t-u) \nu \alpha_k} d\emph{Re}(\beta_k^H)(u) + i \sqrt{\frac{\lambda_k}{2}} \nu^H \int_{-\infty}^t e^{-(t-u) \nu \alpha_k} d\emph{Im}(\beta_k^H)(u) \\
	&=: \hat{\psi}_{k,Re}(t) + i \hat{\psi}_{k,Im}(t)
\end{split}
\end{equation*}
and therefore $\psi(x,t)= \sum_{k \in K} \hat{\psi}_k(t) e_k(x)$, $t \in \mathbb{R}$, $x \in \mathbb{T}^2$. Further, observe that by Assumption \ref{Main-Assumption} we have \\
\begin{equation*}
\begin{split}
(\hat{\psi}_k (t))^* &= \Big( \sqrt{\lambda_k} \nu^H \int_{-\infty}^t e^{-(t-u) \nu \alpha_k} d\beta_k^H(u) \Big)^* = \sqrt{\lambda_k} \nu^H \int_{-\infty}^t e^{-(t-u) \nu \alpha_k} d(\beta_k^H)^*(u) \\
			&= \sqrt{\lambda_{-k}} \nu^H \int_{-\infty}^t e^{-(t-u) \nu \alpha_{-k}} d\beta_{-k}^H(u) = \hat{\psi}_{-k}(t).
\end{split}
\end{equation*}
In particular, for any $s,t \in \mathbb{R}$ and $k,k' \in K$ we obtain 
\begin{equation*}
\mathbb{E} \big(\hat{\psi}_k(t) (\hat{\psi}_{k'}(s))^*  \big) =
\begin{cases}
																								2 \mathbb{E} \big(\hat{\psi}_{k,Re}(t) \hat{\psi}_{k,Re}(s) \big) = 2 \mathbb{E} \big(\hat{\psi}_{k,Im}(t) \hat{\psi}_{k,Im}(s) \big) & \text{if } k=k'\\
																								0 & \text{if } k \neq k'
\end{cases}
\end{equation*}
and 
\begin{equation*}
\mathbb{E} \big( \big\vert \hat{\psi}_k(t) - \hat{\psi}_k(s) \big\vert^2  \big) = 2 \mathbb{E} \big( \big\vert \hat{\psi}_{k,Re}(t) - \hat{\psi}_{k,Re}(s) \big\vert^2  \big) = 2 \mathbb{E} \big( \big\vert \hat{\psi}_{k,Im}(t) - \hat{\psi}_{k,Im}(s) \big\vert^2  \big).
\end{equation*}
Therefore, to compute $\mathbb{E} \big(\hat{\psi}_k(t) (\hat{\psi}_{k'}(s))^*  \big)$ or $\mathbb{E} \big( \big\vert \hat{\psi}_k(t) - \hat{\psi}_k(s) \big\vert^2  \big)$ we only have to compute the associated real part and multiply it by two.
\label{Sec4-Rem1}
\end{rem}
\begin{thm}
Suppose Assumption \ref{Main-Assumption} holds. Further, assume that there is $m \in \mathbb{N}_0$ and $\gamma \in (0,1)$ such that 
\begin{equation}
\sum_{k \in K} \lambda_k \alpha_k^{2 \gamma - 2H} \vert k \vert^{2m} < \infty \quad \text{and} \quad \sum_{k \in K} \lambda_k \alpha_k^{- 2H} \vert k \vert^{2m+2 \gamma} < \infty. 
\label{Sec4-eq5}
\end{equation}
Then there is a unique ergodic mild solution $\psi$ to equation ($M3$). Further, for all $\delta \in \mathbb{N}_0^2$ with $\vert \delta \vert \leq m$ there is a version of $D^\delta \psi$ (again denoted by $D^\delta \psi$) such that 
\begin{equation*}
D^\delta \psi \in C^\epsilon \big(\mathbb{T}^2 \times \mathbb{R} \big)
\end{equation*}
$\mathbb{P}$-a.s. for any $\epsilon \in (0, \min \{ \gamma, H \})$. In particular, $D^\delta \psi$ is real--valued. \\
\label{Sec4-thm2}
\end{thm}
\begin{proof}
By (\ref{Sec4-eq5}) and since Assumption \ref{Main-Assumption} holds Theorem \ref{Sec4-thm1} implies the existence of a unique ergodic mild solution $\psi$ to equation ($M3$). Further, for all $t \in \mathbb{R}$ and $x \in \mathbb{T}^2$ $\psi(x,t)$ is real--valued, since $(\hat{\psi}_k(t))^*= \hat{\psi}_{-k}(t)$ and therefore
\begin{equation*}
\begin{split}
(\psi(x,t))^* &= \Big( \sum_{k \in K} \hat{\psi}_k(t) e^{i<k,x>} \Big)^* = \sum_{k \in K} (\hat{\psi}_k(t))^* (e^{i<k,x>})^* \\
					 &= \sum_{k \in K} \hat{\psi}_{-k}(t) e^{i<-k,x>} = \sum_{k \in K} \hat{\psi}_k(t) e^{i<k,x>} = \psi(x,t).
\end{split}
\end{equation*}
Now let $m \in \mathbb{N}_0$ and $\delta = (\delta_1, \delta_2) \in \mathbb{N}_0^2$ with $\vert \delta \vert = \delta_1 + \delta_2 \leq m$. By (\ref{Sec4-eq5}) it is clear that the stochastic process $(D^\delta \psi(t))_{t \in \mathbb{R}}$, defined by the formal Fourier series
$$ D^\delta \psi(t)= \sum_{k \in K} \sqrt{\lambda_k} \nu^H \int_{- \infty}^t e^{-(t-u) \nu \alpha_k} d\beta_k^H(u) D^\delta e_k, \ t \in \mathbb{R}, $$
is a well--defined $\mathcal{D}((-A)^{\zeta})$--valued stochastic process for some $\zeta \in \mathbb{R}$. But $D^\delta \psi$ is in general a function (and not only a \emph{generalized function}) if $\zeta \geq 0$. However, $ 0 \leq \zeta \leq \gamma$ is already assured by
\begin{equation*}
\mathbb{E} \big( \vert D^\delta \psi(t) \vert^2_{(- A)^\zeta} \big) \leq \sum_{k \in K} \alpha_k^{2 \zeta} \mathbb{E}(\vert \hat{\psi}(t) \vert^2) \vert k \vert^{2m} = \Gamma(2H)H \sum_{k \in K} \lambda_k \alpha_k^{2 \zeta - 2H} \vert k \vert^{2m} < \infty, 
\end{equation*}
where we used $ \vert D^\delta e_k(x) \vert^2 \leq \vert k \vert^{2m} $, Remark \ref{Sec4-Rem1}, Proposition \ref{Section3-Prop1}(i) and (\ref{Sec4-eq5}). Notice also that $D^\delta \psi(x,t)$ is real--valued by using the same argument which leads us to conclude that $\psi(x,t)$ is real--valued. Again, by $ \vert D^\delta e_k(x) \vert^2 \leq \vert k \vert^{2m} $, Remark \ref{Sec4-Rem1}, Proposition \ref{Section3-Prop1}(iii) and (\ref{Sec4-eq5}) we obtain for all $t,s \in \mathbb{R}$ and $x \in \mathbb{T}^2$ 
\begin{equation*}
\begin{split}
\mathbb{E} \big( \vert D^\delta \psi(x,t)- D^\delta \psi(x,s) \vert^2 \big) &= \sum_{k \in K} \mathbb{E} \big( \vert \hat{\psi}_k(t)- \hat{\psi}_k(s) \vert^2 \big) \vert D^\delta e_k(x) \vert^2 \\
		&\leq \sum_{k \in K} \vert k \vert^{2m} \mathbb{E} \big( \vert \hat{\psi}_k(t)- \hat{\psi}_k(s) \vert^2 \big) \\
		&\leq C(H, \nu, \epsilon) \sum_{k \in K} \lambda_k \alpha_k^{2 \epsilon - 2H} \vert k \vert^{2m} \vert t-s \vert^{2 \epsilon} < \infty
\end{split}
\end{equation*} 
for any $ \epsilon \in (0, \min \{ \gamma, H \})$ and some constant $C(H, \nu, \epsilon) >0$. Similarly, using (\ref{Sec4-eq5}) and
$$\vert D^\delta e_k(x) - D^\delta e_k(y) \vert \leq \vert k \vert^m \vert e_k(x) - e_k(y) \vert  \leq C(\eta) \vert k \vert^{m+ \eta} \vert x-y \vert^\eta $$
for any $\eta \in (0,1)$ and a constant $C(\eta)>0$, we get for all $t \in \mathbb{R}$ and $x,y \in \mathbb{T}^2$
\begin{equation*}
\begin{split}
\mathbb{E} \big( \vert D^\delta \psi(x,t)- D^\delta \psi(y,t) \vert^2 \big) &= \sum_{k \in K} \mathbb{E} \big( \vert \hat{\psi}_k(t) \vert^2 \big) \vert D^\delta e_k(x) - D^\delta e_k(y) \vert^2 \\
		&\leq C(\epsilon)  \Gamma(2H)H \sum_{k \in K} \lambda_k \alpha_k^{ - 2H} \vert k \vert^{2m +2 \epsilon} \vert x-y \vert^{2 \epsilon} < \infty
\end{split}
\end{equation*}
for any $ \epsilon \in (0, \gamma)$ and a constant $C(\epsilon)>0$. Therefore, we obtain for all $t,s \in \mathbb{R}$ and $x,y \in \mathbb{T}^2$
\begin{equation*}
\mathbb{E} \big( \vert D^\delta \psi(x,t)-D^\delta \psi(y,s) \vert^2 \big) \leq C(H, \nu, \delta, \epsilon) \big( \vert t-s \vert^{2 \epsilon} + \vert x-y \vert^{2 \epsilon} \big) 
\end{equation*}
for any $ \epsilon \in (0, \min \{ \gamma , H \})$ and a constant $C(H , \nu, \delta, \epsilon) >0$. As $D^\delta \psi(x,t)$ is a normal real--valued random variable, we have
\begin{equation*}
\begin{split}
\mathbb{E} \big( \vert D^\delta \psi(x,t)- D^\delta \psi(y,s) \vert^{2n} \big) &\leq C(H, \nu, \delta, \epsilon, n) \big( \vert t-s \vert^{2\epsilon n} + \vert x-y \vert^{2\epsilon n} \big) \\
		&\leq C(H, \nu, \delta, \epsilon, n) \big( \vert t-s \vert^2 + \vert x-y \vert^2 \big)^{\epsilon n}
\end{split}
\end{equation*}
for all $t,s \in \mathbb{R}$, $x,y \in \mathbb{T}^2$, $n \in \mathbb{N}$, $ \epsilon \in (0, \min \{ \gamma , H \})$ and a constant $ C(H, \nu, \delta, \epsilon, n) >0$. Theorem 3.4 in \cite{DaPrato-Zabczyk-92} implies now that there is a version of $D^\delta \psi$ (again denoted by $D^\delta \psi$) such that $\mathbb{P}$-a.s.
$$ D^\delta \psi \in C^\epsilon \big( \mathbb{T}^2 \times \mathbb{R} \big) $$
for any $\epsilon \in (0, \min \{ \gamma, H \})$.
\end{proof}
\begin{rem}
In \cite{Nualart09} Nualart and Viens established an analogue regularity assertion as in Theorem \ref{Sec4-thm2} for the mild solution of the fractional stochastic heat equation on the circle, but they did not consider the partial derivatives of the mild solution. 
\end{rem}
\begin{cor}
Suppose all assumptions of Theorem \ref{Sec4-thm2} hold. In particular, assume that there is $m \in \mathbb{N}_0$ and $\gamma \in (0,1)$ such that (\ref{Sec4-eq5}) is satisfied and let $\psi$ be the unique ergodic mild solution to equation ($M3$). Then there is a version of $\psi$ (again denoted by $\psi$) such that $\mathbb{P}$-a.s.
\begin{equation*}
\psi \in C(\mathbb{R}, C^{m}(\mathbb{T}^2)).
\end{equation*}
Further, for all $- \infty < T_1 < T_2 < \infty$ there is a positive random variable \\
$K=K(H, \nu, m, \gamma , T_1,T_2): \Omega \to [0, \infty)$ with $ \mathbb{E} \big( K^2 \big) < \infty $
such that $\mathbb{P}$-a.s.
$$ \vert \psi(\omega) \vert_{C([T_1,T_2],C^m (\mathbb{T}^2))} \leq K(\omega). $$
\label{Sec4-Cor1}
\end{cor}
\begin{proof}
Recall from the proof of Theorem \ref{Sec4-thm2} that for any $\delta \in \mathbb{N}_0^2$, $\vert \delta \vert \leq m$, $\epsilon \in (0, \min \{ \gamma, H \})$, $k \in \mathbb{N}$ with $\epsilon k \geq 1$ and $s,t \in \mathbb{R}$, $x=(x_1,x_2),y=(y_1,y_2) \in \mathbb{T}^2$ we have
\begin{equation*}
\begin{split}
\mathbb{E} \big( \vert D^\delta \psi(x,t) - D^\delta \psi(y,s) \vert^{2k} \big) &\leq C(H, \nu, m, \delta, \epsilon, k) \big( \vert t-s \vert^{ \epsilon 2k} + \vert x-y \vert^{ \epsilon 2k} \big) \\
			&\leq \tilde{C}(H, \nu, m, \delta, \epsilon, k) \big( \vert t-s \vert^{ \epsilon 2k} + \vert x_1-y_1 \vert^{ \epsilon 2k} + \vert x_2-y_2 \vert^{ \epsilon 2k} \big)
\end{split}
\end{equation*}
for some constants $C(H, \nu, m, \delta, \epsilon, k), \tilde{C}(H, \nu, m, \delta, \epsilon, k) > 0 $. By Theorem 1.4.1 in \cite{Kunita90} (\emph{Kolmogorov's continuity theorem for random fields}) for $l \in \mathbb{N}$, $\epsilon 2l > 3$, $- \infty < T_1 < T_2 < \infty$ and any
$$ 0 < \beta < \frac{\epsilon 2 l}{2 l} \left( \frac{\frac{3}{3 \frac{1}{\epsilon 2 l}}-3}{\frac{3}{3 \frac{1}{\epsilon 2 l}}} \right) = \epsilon \left( \frac{\epsilon 2 l - 3}{\epsilon 2l} \right) < \min \{ \gamma, H \}, $$
there is a positive random variable $K=K(H, \nu, m, \delta, \epsilon, \beta, l, T_1, T_2): \Omega \to [0, \infty)$
with 
\begin{equation}
\mathbb{E} \big(K^{2l} \big)< \infty 
\label{Sec4-eq7}
\end{equation}
and a version of $\psi$ (again denoted by $\psi$) such that for all $t,s \in [T_1, T_2]$ and $x=(x_1,x_2),y=(y_1,y_2) \in \mathbb{T}^2$ we have $\mathbb{P}$-a.s. 
\begin{equation*}
\vert D^\delta \psi(x,t)(\omega) - D^\delta \psi(y,s)(\omega) \vert \\ 
			\leq K(\omega) \big( \vert t-s \vert^\beta + \vert x_1-y_1 \vert^\beta + \vert x_2-y_2 \vert^\beta \big).
\end{equation*}
In particular, for $t,t_0 \in [T_1,T_2]$ and $x=(x_1,x_2),x_0=(x_{0,1}, x_{0,2}) \in \mathbb{T}^2$ we have $\mathbb{P}$-a.s. 
\begin{equation*}
\begin{split}
\vert D^\delta \psi(x,t)(\omega) \vert &\leq \vert D^\delta \psi(x,t)(\omega) - D^\delta \psi(x_0, t_0)(\omega) \vert + \vert D^\delta \psi(x_0, t_0)(\omega) \vert \\
	&\leq K(\omega) \big( \vert t-t_0 \vert^\beta + \vert x_1-x_{0,1} \vert^\beta + \vert x_2-x_{0,2} \vert^\beta \big) + \vert D^\delta \psi(t_0,x_0)(\omega) \vert
\end{split}
\end{equation*}
and therefore $\mathbb{P}$-a.s.
\begin{equation}
\sup_{t \in [T_1,T_2]} \sup_{x \in \mathbb{T}^2} \vert D^\delta \psi(x,t)(\omega) \vert \leq K(\omega) C(\beta, T_1, T_2) + \vert D^\delta \psi(x_0, t_0)(\omega) \vert
\label{Sec4-eq8}
\end{equation}
for some constant $ C(\beta, T_1, T_2) >0$. The assertions of the corollary follow now by (\ref{Sec4-eq7}), (\ref{Sec4-eq8}) and since we have
\begin{equation*}
\begin{split}
\mathbb{E} \big( \vert  D^\delta \psi(x_0, t_0) \vert^2 \big) &= \sum_{k \in K} \lambda_k \nu^{2H} \mathbb{E} \left( \big\vert \int_{- \infty}^{t_0} e^{-(t_0-u) \nu \alpha_k} d\beta_k^H(u)  \big\vert^2 \right) \vert  D^\delta e_k(x_0) \vert^2 \\
			&\leq \Gamma(2H)H \sum_{k \in K} \frac{\lambda_k }{ \alpha_k^{2H}} \vert k \vert^{2m} < \infty
\end{split}
\end{equation*}
for every $x_0 \in \mathbb{T}^2$, $t_0 \in \mathbb{R}$ and $\delta \in \mathbb{N}_0^2$, $ \vert \delta \vert \leq m$, where we used (\ref{Sec4-eq5}), Proposition \ref{Section3-Prop1}(i) and Remark \ref{Sec4-Rem1}.
\end{proof}
We apply Corollary \ref{Sec4-Cor1} to state an existence and uniqueness result for the transport equation ($M1$). Consider the differential equation ($M1$) as a first order system
\begin{equation}
\frac{d}{dt} \binom{x(t)}{\dot{x}(t)} = f_{\psi(\omega), \tau}(t, (x(t), \dot{x}(t))), \ \binom{x(0)}{\dot{x}(0)}= \binom{\overline{x}}{\overline{y}} \in \mathbb{T}^2 \times \mathbb{R}^2,
\label{Sec4-eq9}
\end{equation}
where $\tau > 0$ and $f_{\psi(\omega), \tau}: \mathbb{R} \times \mathbb{T}^2 \times \mathbb{R}^2 \rightarrow \mathbb{R}^4$ is defined by
\begin{equation}
(t, x, y ) \mapsto  f_{\psi(\omega), \tau}(t, x, y) = \begin{pmatrix}
	y \\
	\frac{1}{\tau} (\nabla^{\bot} \psi(x,t)(\omega) -y)
\end{pmatrix}. 
\label{Sec4-eq10}
\end{equation}
Here $\psi(\cdot, \cdot)(\omega)$ with $\omega \in \Omega$ denotes a realization of the ergodic mild solution of ($M3$).\\
We say that ($M1$) has a \emph{unique local} $C^m$\emph{--solution} $\mathbb{P}$-a.s. for some $m \in \mathbb{N}$ if for all $(\overline{x}, \overline{y}) \in \mathbb{T}^2 \times \mathbb{R}^2$ there is an open interval $I(\omega) \subseteq \mathbb{R}$ including $0$ and a function 
$$ \binom{x(\cdot)}{\dot{x}(\cdot)} \in C^1 \left(I(\omega), C^m \big(\mathbb{T}^2 \times \mathbb{R}^2, \mathbb{T}^2 \times \mathbb{R}^2 \big) \right), $$
which satisfies uniquely the equation (\ref{Sec4-eq9}) for all $t \in I(\omega)$ $\mathbb{P}$-a.s.. We say that ($M1$) has a \emph{unique global} $C^m$\emph{--solution} $\mathbb{P}$-a.s. for some $m \in \mathbb{N}$, if ($M1$) has a unique local $C^m$--solution with $I(\omega)= \mathbb{R}$ $\mathbb{P}$-a.s..
\begin{cor}
Suppose Assumption \ref{Main-Assumption} holds. Further, assume that there is $m \in \mathbb{N}$, $m \geq 2$, and $\gamma \in (0,1)$ such that 
\begin{equation}
\sum_{k \in K} \lambda_k \alpha_k^{2 \gamma - 2H} \vert k \vert^{2m} < \infty \quad \text{and} \quad \sum_{k \in K} \lambda_k \alpha_k^{- 2H} \vert k \vert^{2m+2 \gamma} < \infty. 
\label{Sec4-eq11}
\end{equation}
Then ($M1$) has a unique global $C^{m-1}$--solution $\mathbb{P}$-a.s..
\label{Sec4-Cor2}
\end{cor}
\begin{proof}
By Corollary \ref{Sec4-Cor1} there is a version of the strictly stationary solution $\psi$ to equation ($M3$) (again denoted by $\psi$) such that $\psi \in C(\mathbb{R},C^{m}(\mathbb{T}^2))$ $\mathbb{P}$-a.s.. This implies that $f_{\psi(\omega), \tau} \in C(\mathbb{R}, C^{m-1}(\mathbb{T}^2 \times \mathbb{R}^2, \mathbb{R}^4))$ where $f_{\psi(\omega), \tau}$ is defined in (\ref{Sec4-eq10}). Therefore, see e.g. Appendix B in \cite{Arnold98}, ($M1$) has a unique local $C^{m-1}$--solution $\mathbb{P}$-a.s.. To establish $\mathbb{P}$-a.s. global solutions we have to find locally integrable, positive functions
$ \alpha_\omega, \beta_\omega : \mathbb{R} \to [0, \infty)$ which may depend upon the realization $\omega \in \Omega$, such that $\mathbb{P}$-a.s.
$$ \vert f_{\psi(\omega), \tau}(t, x, y) \vert \leq \alpha_\omega(t) \vert (x,y) \vert + \beta_\omega(t). $$
By Lemma \ref{Sec6-Lem1}(ii) (see Section 6 below) there is a constant $K(\omega)>0$ such that $\mathbb{P}$-a.s.
\begin{equation*}
 \big\vert \frac{ \partial}{ \partial x_1} \psi(x,t)(\omega) \big\vert^2 + \big\vert \frac{ \partial}{ \partial x_2} \psi(x,t)(\omega) \big\vert^2 \leq \vert \psi(t) (\omega) \vert_{C^1(\mathbb{T}^2)}^2 \leq \left( \vert t \vert + K(\omega) \right)^2
\end{equation*}
for all $t \in \mathbb{R}$. Hence, we have for all $t \in \mathbb{R}$, $x=(x_1,x_2) \in \mathbb{T}^2$ and $y=(y_1,y_2) \in \mathbb{R}^2$
\begin{equation*}
\begin{split}
\vert f_{\psi(\omega), \tau}(t,x,y)\vert^2 &= y_1^2+ y_2^2 + \big\vert \frac{ \partial}{ \partial x_2} \psi(x,t)(\omega) - y_1 \big\vert^2 + \big\vert - \frac{ \partial}{ \partial x_1} \psi(x,t)(\omega) - y_2 \big\vert^2 \\
			&\leq 3 (y_1^2+ y_2^2)+ 2 \left( \big\vert \frac{ \partial}{ \partial x_1} \psi(x,t)(\omega) \big\vert^2 + \big\vert \frac{ \partial}{ \partial x_2} \psi(x,t)(\omega) \big\vert^2 \right) \\
			&\leq 3 \vert (x_1,x_2,y_1,y_2) \vert^2 + 2 \left( \vert t \vert + K(\omega) \right)^2
\end{split}
\end{equation*}
$\mathbb{P}$-a.s. and in particular
\begin{equation*}
\vert f_{\psi(\omega), \tau}(t,x,y) \vert \leq \sqrt{3} \vert (x, y) \vert + \sqrt{2} \left( \vert t \vert + K(\omega) \right).
\end{equation*}
\end{proof}


\section{Matching desired statistical properties of the velocity field}

We suppose that Assumption \ref{Main-Assumption} holds and set $A=\Delta$ (and thereby $\alpha_k= \vert k \vert^2$, $k \in K$). Further, assume that 
\begin{equation}
\sum_{k \in K} \lambda_k \vert k \vert^{2+4 \gamma-4H} < \infty 
\label{Sec5-eq1}
\end{equation}
for some $\gamma >0$. So by Corollary \ref{Sec4-Cor1} there is a unique ergodic mild solution $\psi$ of ($M3$) and there is a version of $\psi$ (again denoted by $\psi$) such that $\mathbb{P}$-a.s. $ \psi \in C \left(\mathbb{R}, C^1(\mathbb{T}^2) \right)$. \\
Since $v=\nabla^{\bot} \psi$, we have
\begin{equation}
v(x,t)= \binom{v_1(x,t)}{v_2(x,t)} = \sum_{k \in K} i \binom{k_2}{-k_1} \hat{\psi}_k(t) e_k(x) = \sum_{k \in K} \hat{v}_k(t) e_k(x), \ t \in \mathbb{R}, \ x \in \mathbb{T}^2,
\label{Sec5-eq2}
\end{equation}
where we set $\hat{v}_k(t) := i \binom{k_2}{-k_1} \hat{\psi}_k(t)$. Since $(\hat{\psi}_k(t))_{t \in \mathbb{R}}$, $k \in K$, are mean zero Gaussian processes, independent, except $(\hat{\psi}_k)^* = \hat{\psi}_{-k}$, $v$ is a \emph{mean zero Gaussian random field}. The \emph{autocovariance function} $ R : \mathbb{T}^2 \times \mathbb{T}^2  \times \mathbb{R} \times \mathbb{R} \rightarrow \mathbb{R}^{2 \times 2}$ of $v$ is given by
\begin{equation*}
\begin{split}
R(x,y,t,s) &= \left( \mathbb{E} \big( v_i(x,t) v_j(y,s)  \big) \right)_{1 \leq i,j \leq 2} \\
					 &=  \sum_{k \in K} \begin{pmatrix} 
	k_2^2 & - k_2 k_1 \\ - k_1 k_2 & k_1^2
									\end{pmatrix} \lambda_k \vert k \vert^{-4H}  \delta_k(t-s;H,\nu) \ e_k(x-y),
\end{split}
\end{equation*}
where we set
$$ \delta_k(t-s; H, \nu) := \frac{\Gamma(2H+1) \sin(\pi H)}{\pi} \int_{0}^{\infty} \cos( (t-s) \nu \vert k \vert^2 z) \frac{\vert z \vert}{1 + z^2} dz$$
and where we applied Proposition \ref{Section3-Prop1}(i) and Remark \ref{Sec4-Rem1}. Therefore, $v$ is \emph{stationary} and \emph{homogeneous}. For $k \in K$, the energy of the Fourier mode $k$ is defined by
\begin{equation*}
\begin{split}
\mathcal{E}(k)&= \frac{1}{2} \mathbb{E} \left( \hat{v}_k(t) (\hat{v}_k(t))^* \right) = \frac{1}{2} \mathbb{E} \left( \vert \hat{v}_k(t) \vert^2 \right)=  \frac{1}{2} \vert k \vert^2 \mathbb{E} \left( \vert \hat{\psi}_k(t) \vert^2 \right) = \frac{1}{2} \vert k \vert^2 \mathbb{E} \left( \vert \hat{\psi}_k(0) \vert^2 \right) \\
			&= \frac{\Gamma(2H)H}{2} \lambda_k \vert k \vert^{2-4H},
\end{split}
\end{equation*}
where we again applied Proposition \ref{Section3-Prop1}(i) and Remark \ref{Sec4-Rem1}. \\
To ensure \emph{isotropy}, we set $\lambda_k:= \zeta(\vert k \vert)$, $k \in K$, for a suitable positive function $\zeta: [0, \infty) \to [0, \infty)$, but keep in mind that (\ref{Sec5-eq1}) should still be satisfied. Then the energy of a Fourier mode $k \in K$ depends only on the length of $k$, in that, $\mathcal{E}(\kappa) :=\mathcal{E}(k)=\mathcal{E}(k')$ whenever $\kappa=\vert k \vert= \vert k' \vert $ for $k, k' \in K$. In such an isotropic random field of the form (\ref{Sec5-eq2}) it is customary to define the \emph{energy spectrum} in terms of total energy in all the Fourier modes of the same length $\kappa= \vert k \vert$ by $ E(\kappa) := \# \{ k \in K \big| \ \vert k \vert = \kappa \} \ \mathcal{E}(\kappa)$. Clearly, $E(\kappa)$ can be approximated by $E(\kappa) \approx C \kappa \mathcal{E}(\kappa)$ with some constant $C>0$. \\
In general the energy spectrum $E(\cdot)$ can be devided in three ranges: For small $\vert k \vert$, where the energy is injected, $E(\cdot)$ increases in $\vert k \vert $ algebraically. For large $\vert k \vert $, where the energy dissipates, we just set $\lambda_k$ and therefore the energy spectrum to zero (\emph{ultraviolet cut--off}). For intermediate $\vert k \vert $, in the so--called \emph{inertial subrange}, $E(\cdot)$ decays in $\vert k \vert $ algebraically. \\
The spectrum $(\lambda_k)_{k \in K}$ of $Q$ can be chosen so that the energy spectrum of $v$ matches experimentally observed energy spectra of a turbulent fluid flow, e.g. the \emph{Kolmogorov spectrum}: $E( \vert k \vert) \varpropto \vert k \vert^{- \frac{5}{3}}$ and therefore $\mathcal{E}( \vert k \vert)\varpropto \lambda_k \vert k \vert^{2-4H} \varpropto \vert k \vert^{- \frac{8}{3}}$, i.e. $ \lambda_k \varpropto \vert k \vert^{- \frac{14}{3}+ 4H}$. \\
Obviously, the decay of the spectrum $(\lambda_k)_{k \in K}$ of $Q$ as $\vert k \vert \to \infty$ determines the regularity of the velocity field $v$ and by this also the regularity of the transport equation ($M1$). From the physical point of view by applying Corollary \ref{Sec4-Cor2} it is clear that for any spectrum in the inertial subrange we have a unique global solution to ($M1$) due to the ultraviolet cut--off. From the mathematical point of view it is interesting to ask whether there is a (unique) solution to ($M1$) if we match the spectrum $(\lambda_k)_{k \in K}$ for all modes without the cut--off. In view of the Kolmogorov spectrum and (\ref{Sec4-eq5}) in Corollary \ref{Sec4-Cor1} we have
$$ \vert k \vert^{- \frac{14}{3}+4H-4H+2m+4 \gamma}= \vert k \vert^{-2} \vert k \vert^{- \frac{8}{3}+2m+4 \gamma} $$
and $ - \frac{8}{3}+2m+4 \gamma < 0 $ is satisfied for $m=1$ and $\gamma < \frac{1}{6}$. So by Corollary \ref{Sec4-Cor1} there is a version of $\psi$ (again denoted by $\psi$) such that $\mathbb{P}$-a.s. $ \psi \in C \left(\mathbb{R}, C^1(\mathbb{T}^2) \right)$ and  $v=\nabla^{\bot} \psi \in C \left(\mathbb{R}, C(\mathbb{T}^2, \mathbb{R}^2) \right)$. Hence, by the classic Peano existence theorem there is a solution to ($M1$), but uniqueness may fail. \\
Further, recall that our main motivation to use fractional noise was to match the statistical property
$$ \mathbb{E} \left( \vert v(x,t)- v(x,s) \vert^2 \right) \sim C \ \vert t-s \vert^{2H} $$
for $t,s \geq 0$, $x \in \mathbb{T}^2$ and some constant $C>0$. We have the following result.
\begin{prop}
Suppose Assumption \ref{Main-Assumption} holds and that there is $m \in \mathbb{N}$ and $\epsilon > 0$ such that
	\begin{equation}
		\sum_{k \in K} \lambda_k \vert k \vert^{2m + \epsilon} < \infty.
	\label{Sec5-eq3}
	\end{equation}
Then there is a unique ergodic mild solution $\psi$ to equation ($M3$) and there is a version of $\psi$ (again denoted by $\psi$) such that $\mathbb{P}$-a.s. $ \psi \in C \left(\mathbb{R}, C^m(\mathbb{T}^2) \right)$ and  $v=\nabla^{\bot} \psi \in C \left(\mathbb{R}, C^{m-1}(\mathbb{T}^2, \mathbb{R}^2) \right)$. Further, there is a constant $C(H, \nu)>0$ such that for any $t,s \in \mathbb{R}$ and $x \in \mathbb{T}^2$ we have
\begin{equation}
\mathbb{E} \left( \vert v(x,t)- v(x,s) \vert^2 \right) \leq C(H, \nu) \vert t-s \vert^{2H}
\label{Sec5-eq4}
\end{equation}
and for a fixed $T>0$ there is a constant $C(H, \nu, T)>0$ such that for any $t,s \in [-T,T]$ and $x \in \mathbb{T}^2$ we have
\begin{equation}
C(H, \nu, T) \vert t-s \vert^{2H} \leq \mathbb{E} \left( \vert v(x,t)- v(x,s) \vert^2 \right).
\label{Sec5-eq5}
\end{equation}
\label{Sec5-Prop1}
\end{prop}
\begin{proof}
The first assertion follows by Corollary \ref{Sec4-Cor1} in view of (\ref{Sec5-eq3}). Further, by (\ref{Sec5-eq2}) we have 
\begin{equation*}
\begin{split}
\mathbb{E} \left( \vert v(x,t)- v(x,s) \vert^2 \right) &= \sum_{k \in K} \vert k \vert^2 \mathbb{E} \left( \vert \hat{\psi}_k(t)- \hat{\psi}_k(s) \vert^2 \right) \vert e_k(x) \vert^2 \\ 
	&= \sum_{k \in K} \vert k \vert^2 \mathbb{E} \left( \vert \hat{\psi}_k(t)- \hat{\psi}_k(s) \vert^2 \right) 
\end{split}
\end{equation*}
for any $t,s \in \mathbb{R}$ and $x \in \mathbb{T}^2$. (\ref{Sec5-eq4}) and (\ref{Sec5-eq5}) follow now by (\ref{Sec5-eq3}), Proposition \ref{Section3-Prop1}(ii),(iii) and Remark \ref{Sec4-Rem1}.
\end{proof}
\begin{rem}
It should be noted that (\ref{Sec5-eq3}) in Proposition \ref{Sec5-Prop1} is very restrictive and not satisfied for $H \in [\frac{1}{6},1)$ if $m=1$ and if we use the Kolmogorov spectrum. But in view of the ultraviolet cut--off in the region where the energy dissipates, (\ref{Sec5-eq3}) is fulfilled for any energy spectrum in the inertial subrange.
\end{rem}


\section{The model as random dynamical system and existence of the random pullback attractor}

First we recall some required definitions from the theory of random dynamical systems. For the general theory of random dynamical systems we refer to the excellent monograph \cite{Arnold98}. \\
In the following $(X,d)$ is a complete separable metric space and $2^{X}$ denotes the set of all subsets of $X$. Further, for $B \in 2^{X}$ we denote by $\overline{B}$ the closure of $B$ in $X$ and by $B^c:= B \setminus X $ the complement of $B$ in $X$. For $x \in X$ and $B,C \in 2^X$ we define the \emph{semidistance} by
$$ dist(x,B):= \inf_{b \in B} d(x,b) \quad \text{and} \quad dist(B,C):= \sup_{b \in B} \inf_{c \in C} d(b,c). $$
We make the convention $d(x, \emptyset)= \infty$, where $\emptyset$ denotes the empty set. \\
\noindent
A family $(\theta(t))_{t \in \mathbb{R}})$ of mappings on $\Omega$ into itself is called a \emph{metric dynamical system} and is defined by $ \left(\Omega,\mathcal{F}, \mathbb{P}, (\theta(t))_{t \in \mathbb{R}} \right)$ if it satisfies the following four conditions:
\begin{enumerate}
	\item[(i)] The mapping $(\omega, t) \mapsto \theta(t)\omega$ is $\mathcal{F} \otimes \mathcal{B}(\mathbb{R}) $--$\mathcal{F}$ measurable.
	\item[(ii)] $\theta(0)= id_{\Omega} = \text{ identity map in } \Omega$.
	\item[(iii)] $(\theta(t))_{t \in \mathbb{R}}$ satisfies the \emph{flow property}, i.e. $ \theta(t+s)= \theta(t) \circ \theta(s)$ for all $s,t \in \mathbb{R}$, where $\circ$ denotes the composition.
	\item[(iv)] $(\theta(t))_{t \in \mathbb{R}}$ is a family of measure preserving transformations, i.e. $ \mathbb{P} (\theta(t)^{-1}(A))= \mathbb{P}(A)$ for all $A \in \mathcal{F}$ and $t \in \mathbb{R}$, where $\theta(t)^{-1}(A):= \{ \omega \in \Omega \ | \ \theta(t)\omega \in A \}.$
\end{enumerate}
We say that a metric dynamical system is \emph{ergodic} if for all $A \in \mathcal{F}$ such that $\theta(t)^{-1}(A)=A$ for all $t \in \mathbb{R}$, we have $\mathbb{P}(A) \in \{0,1 \}$. \\
A \emph{random dynamical system} (RDS) on $X$ over a metric dynamical system \\
$(\Omega,\mathcal{F}, \mathbb{P}, (\theta(t))_{t \in \mathbb{R}})$ with time $\mathbb{R}$ is a mapping
$$ \phi: \mathbb{R} \times \Omega \times X \to X, \ (t, \omega, x) \mapsto \phi(t, \omega,x),$$
with the following properties:
\begin{enumerate}
	\item[(i)] \emph{Measurability}: $\phi$ is $\mathcal{B}(\mathbb{R})\otimes \mathcal{F} \otimes \mathcal{B}(X)$--$\mathcal{B}(X)$ measurable.
	\item[(ii)] \emph{Cocycle property}: The mappings $ \phi(t,\omega):= \phi(t, \omega, \cdot): X \to X $
	form a \emph{cocycle over} $\theta$, i.e. they satisfy $\phi(0, \omega) = id_X$  for all  $\omega \in \Omega $ and	$\phi(t+s, \omega) = \phi(t, \theta(s) \omega) \circ \phi(s, \omega)$  for all $s,t \in \mathbb{R}, \omega \in \Omega$. Here $\circ$ denotes the composition.  
\end{enumerate}
We call a RDS $\phi$ \emph{continuous} or $C^0$--RDS if $ (t,x) \mapsto \phi(t, \omega, x)$ is continuous for every $\omega \in \Omega$ and we say that a RDS $\phi$ is a $C^k$--RDS, where $1 \leq k \leq \infty$ if for each $(t, \omega) \in \mathbb{R} \times \Omega$ the mapping $ x \mapsto \phi(t, \omega, x)$ is $k$--times differentiable w.r.t. $x \in X$ and the derivatives are continuous w.r.t. $(t, x) \in \mathbb{R} \times X$ for each $\omega \in \Omega$. \\
Because of the non--autonomous noise dependence of a RDS, generalized concepts of absorption, attraction and invariance of (random) sets have to be introduced. For that we also recall some facts from the theory of \emph{measurable (closed) random sets}, sometimes also called \emph{measurable multifunctions}, see \cite{Crauel02}. \\
A set valued map $D: \Omega \to 2^{X}$ taking values in closed subsets of $X$ is said to be \emph{measurable} if for each $x \in X $ the map $ \omega \mapsto dist(x, D(\omega))$ is $\mathcal{F}$--$\mathcal{B}([0,\infty))$ measurable. In this case $D$ is called a \emph{closed random set (of X)}. \\
A \emph{universe of closed random sets $\mathcal{D}$ (of X)} is a system of non--empty closed random sets of $X$, such that $\mathcal{D}$ is closed under inclusion, i.e. if $D$ and $ D'$ are non--empty closed random sets of X, such that $D'(\omega) \subseteq D(\omega)$, for all $ \omega \in \Omega,$ and $D \in \mathcal{D}$, then $D' \in \mathcal{D}$. \\
Now let $\phi$ be a RDS and $\mathcal{D}$ a universe of closed random sets. \\
A closed random set $B$ is called \emph{(strictly) $\phi$--forward invariant} if
$$\phi(t, \omega)B(\omega) \subseteq B(\theta(t) \omega) \quad (\phi(t, \omega)B(\omega) = B(\theta(t) \omega))$$
for all $\omega \in \Omega$, $t \geq 0$. \\
A closed random set $B \in \mathcal{D}$ is called \emph{$\mathcal{D}$--absorbing} if for any $D \in \mathcal{D}$, $\omega \in \Omega$ there exists a time $t_D(\omega) \geq 0$, the so--called \emph{absorption time}, such that for any $t > t_D(\omega)$
$$ \phi(t, \theta(-t) \omega)D(\theta(-t) \omega) \subseteq B(\omega). $$
A closed random set $A \in \mathcal{D}$ with compact values is called \emph{random} $\mathcal{D}$--\emph{attractor} of the RDS $\phi$ if $A$ is strictly $\phi$--forward invariant and for any $\omega \in \Omega$ we have
\begin{equation*}
dist \left(\overline{\phi(t, \theta(-t) \omega) D(\theta(-t) \omega)}, A(\omega) \right) \to 0 \text{ as } t \to \infty
\end{equation*}
for any $D \in \mathcal{D}$. \\
\noindent
If there exists a random $\mathcal{D}$--attractor, then the attractor is already unique. Indeed, suppose we have two attractors $A_i \in \mathcal{D}$, $i=1,2$. It follows for any $\omega \in \Omega$ that
$$ dist(A_1(\omega),A_2(\omega))= \lim_{t \to \infty} dist \left(\overline{\phi(t, \theta(-t) \omega) A_1(\theta(-t) \omega)}, A_2(\omega) \right)=0. $$
Therefore, $A_1(\omega) \subseteq A_2(\omega)$ for any $\omega \in \Omega$. Similarly, we can find the contrary inclusion. So the $\mathcal{D}$--attractor is unique.
\begin{thm}
Let $\phi$ be a continuous RDS and $\mathcal{D}$ a universe of closed random sets. In addition, we assume the existence of a $\phi$--forward invariant and $\mathcal{D}$--absorbing closed random set $B$ with compact values. Then the RDS has a unique random $\mathcal{D}$--attractor given by 
$$ A(\omega)= \bigcap_{t \in \mathbb{N}} \phi(t, \theta(-t) \omega)B(\theta(-t)\omega). $$
\label{Sec6-Thm1}
\end{thm}
\begin{proof}
See Proposition 9.3.2 in \cite{Arnold98} or Theorem 2.4 in \cite{Schmalfuss97}.
\end{proof}
\begin{rem}
The existence of a random attractor for a RDS $\phi$ implies the existence of an \emph{invariant measure} for $\phi$, see \cite{Crauel99} for the definition and more details. In particular, all invariant measures are supported on the random attractor.
\end{rem}
Now we come back to our model.
\begin{prop}
Suppose Assumption \ref{Main-Assumption} holds and that there is $m \in \mathbb{N}_0$ and $\gamma \in (0,1)$ such that
\begin{equation}
\sum_{k \in K} \lambda_k \alpha_k^{2 \gamma - 2H} \vert k \vert^{2m} < \infty \quad \text{and} \quad \sum_{k \in K} \lambda_k \alpha_k^{- 2H} \vert k \vert^{2m+2 \gamma} < \infty. 
\label{Sec6-eq1}
\end{equation}
Then the quadruple $\left(\Omega, \mathcal{F}, \mathbb{P}, (\theta(t))_{t \in \mathbb{R}} \right)$ defines an ergodic metric dynamical system, where
\begin{itemize}
	\item $\Omega = C \left(\mathbb{R}, C^m(\mathbb{T}^2) \right)$ equipped with the compact open topology given by the complete metric $d(\psi, \widetilde{\psi}):= \sum_{n=1}^{\infty} \vert \psi - \widetilde{\psi} \vert_n / \big(2^n (1+ \vert \psi - \widetilde{\psi} \vert_n) \big)$, where $\vert \psi - \widetilde{\psi} \vert_n := \sup_{-n \leq t \leq n} \vert \psi(t) - \widetilde{\psi}(t) \vert_{C^m(\mathbb{T}^2)}$,
	\item $\mathcal{F}$ is the associated Borel $\sigma$--algebra, which is the trace in $\Omega$ of the product $\sigma$--algebra $(\mathcal{B}(C^m(\mathbb{T}^2)))^{\otimes \mathbb{R}}$,
	\item $\mathbb{P}$ is the distribution of the ergodic mild solution of ($M3$),
	\item $(\theta(t))_{t \in \mathbb{R}}$ is the group of shifts, i.e. $ \theta(t) \psi(s)= \psi(t+s)$ for all $t,s \in \mathbb{R}$ and $\psi \in \Omega$.
\end{itemize}
\label{Sec6-Prop1}
\end{prop}
\begin{proof}
By Corollary \ref{Sec4-Cor1} and assumptions (\ref{Sec6-eq1}) the ergodic mild solution $\psi$ to equation ($M3$) is realized on $\Omega:= C \left(\mathbb{R}, C^m(\mathbb{T}^2) \right)$. Endowing $\Omega$ with the compact open topology makes $\Omega$ a Polish space, actually a Frechet space. \\
The group of shifts $(\theta(t))_{t \in \mathbb{R}}$ on $\Omega$ defined by $ \theta(t)\psi(\cdot)= \psi(t+ \cdot) $
for $t \in \mathbb{R}$ satisfies the flow property and is measure preserving, since $\mathbb{P}$ is the distribution of the ergodic mild solution of the equation ($M3$). \\
Obviously, $t \mapsto \theta(t)\psi$ is continuous for all $\psi \in \Omega$ and $ \psi \mapsto \theta(t) \psi$ is continuous for all $t \in \mathbb{R}$. Therefore, by Lemma 2.1 in \cite{Crauel02} $(t, \psi) \mapsto \theta(t) \psi$ is $ \mathcal{B}( \mathbb{R}) \otimes \mathcal{F}$--$\mathcal{F}$ measurable. \\ 
Hence, $(\Omega, \mathcal{F}, \mathbb{P}, (\theta(t))_{t \in \mathbb{R}})$ defines an ergodic metric dynamical system.
\end{proof}
The next lemma will be used several times in this section.
\begin{lem}
Suppose all assumptions of Proposition \ref{Sec6-Prop1} hold. In particular assume that there is $m \in \mathbb{N}_0$ and $\gamma \in (0,1)$ such that (\ref{Sec6-eq1}) is satisfied and let $ \left(\Omega, \mathcal{F}, \mathbb{P}, (\theta(t))_{t \in \mathbb{R}} \right)$ be the ergodic metric dynamical system introduced in Proposition \ref{Sec6-Prop1}. Then the following assertions are valid.
\begin{enumerate}
	\item[(i)] For all $- \infty < T_1 < t < T_2 < \infty $ the mappings $ \psi \mapsto \vert \psi(t) \vert_{C^m(\mathbb{T}^2)} $ and $ \psi \mapsto \vert \psi \vert_{C([T_1,T_2], C^m(\mathbb{T}^2))}$ are $\mathcal{F}$--$\mathcal{B}(\mathbb{R})$ measurable and there is a constant \\
$C=C(m, H, \nu, \gamma, T_1, T_2) >0$ such that
\begin{equation}
\mathbb{E} \left( \vert id_\Omega(t) \vert_{C^m(\mathbb{T}^2)}^2 \right) \leq \mathbb{E} \left( \vert id_\Omega \vert_{C([T_1,T_2], C^m(\mathbb{T}^2))}^2 \right) \leq C,
\label{Sec6-eq2}
\end{equation}	
where $id_\Omega : \Omega \rightarrow \Omega$, $ \psi \mapsto \psi$.
		\item[(ii)] There is a $( \theta(t))_{t \in \mathbb{R}}$--invariant set $\mathcal{F} \ni \Omega_0 = \Omega_0(m, \gamma, H, \nu) \subset \Omega$ (i.e. $\theta(t)^{-1} \Omega_0 = \Omega_0 $ for all $t \in \mathbb{R}$) with $\mathbb{P} (\Omega_0) =1$ such that for all $\delta > 0$, $\psi \in \Omega_0$ there is a constant $C(\psi)=C(\psi, m, \gamma, H, \nu, \delta) >0$ such that 
\begin{equation}
\vert \psi(t) \vert_{C^m(\mathbb{T}^2)} \leq \delta \vert t \vert + C(\psi)
\label{Sec6-eq3}
\end{equation}
for all $t \in \mathbb{R}$. In particular, the mapping
$$ \psi \mapsto \begin{cases}
															 \int_{- \infty}^0 e^{\frac{s}{\tau}} \vert \psi(s) \vert_{C^m(\mathbb{T}^2)}^2 ds  & \text{ for } \psi \in \Omega_0 \\
																								0 & \text{ for } \psi \notin \Omega_0
								\end{cases} $$
with $\tau > 0$ is well--defined and $\mathcal{F}$--$\mathcal{B}(\mathbb{R})$ measurable. \\
\end{enumerate}
\label{Sec6-Lem1}
\end{lem}
\begin{proof}
(i): Since $\Omega$ is endowed with the compact open topology, the mappings $ \psi \mapsto \vert \psi(t) \vert_{C^m(\mathbb{T}^2)} $ and $ \psi \mapsto \vert \psi \vert_{C([T_1,T_2], C^m(\mathbb{T}^2))}$ for all $- \infty < T_1 < t < T_2 < \infty $ are $\mathcal{F}$--$\mathcal{B}(\mathbb{R})$ measurable. Further, by Corollary \ref{Sec4-Cor1} for any $- \infty < T_1 < T_2 < \infty$ there is a positive random variable $K(T_1,T_2)=K(m, \gamma, H, \nu, T_1, T_2) \in L^2(\Omega, \mathcal{F}^{\mathbb{P}}, \mathbb{P}) \subseteq L^1(\Omega, \mathcal{F}^{\mathbb{P}}, \mathbb{P})$, where $\mathcal{F}^{\mathbb{P}}$ denotes the completion of $\mathcal{F}$ w.r.t. $\mathbb{P}$ such that $\mathbb{P}$-a.s.
\begin{equation}
\vert \psi \vert_{C([T_1,T_2], C^m(\mathbb{T}^2))} \leq K(T_1,T_2, \psi), 
\label{Sec6-eq4}
\end{equation}
since $K$ may be measurable only with respect to the completed $\sigma$--algebra $\mathcal{F}^{\mathbb{P}}$. So for all $- \infty < T_1 < t < T_2 < \infty $ we have
\begin{equation}
\begin{split}
\mathbb{E} \big(\vert id_\Omega(t) \vert_{C^m(\mathbb{T}^2)} \big) \leq \mathbb{E} \big(\vert id_\Omega \vert_{C([T_1,T_2], C^m(\mathbb{T}^2))} \big) &= \mathbb{E}^\mathbb{P} \big( \vert id_\Omega \vert_{C([T_1,T_2], C^m(\mathbb{T}^2))} \big) \\ 				&\leq \mathbb{E}^{\mathbb{P}} \big( K(T_1,T_2) \big) < \infty,
\end{split}
\label{Sec6-eq5}
\end{equation}
where $\mathbb{E}^{\mathbb{P}}$ is related to the extension of $\mathbb{P}$ to $\mathcal{F}^{\mathbb{P}}$. \\
(ii): We have by (\ref{Sec6-eq4})
\begin{equation}
\sup_{r \in [0,1]} \vert \theta(r) \psi \vert_{C([T_1,T_2], C^m(\mathbb{T}^2))} = \vert \psi \vert_{C([T_1,T_2+1], C^m(\mathbb{T}^2))} \leq K(T_1,T_2+1, \psi)
\label{Sec6-eq6}
\end{equation}
and
\begin{equation}
\begin{split}
\mathbb{E} \big( \sup_{r \in [0,1]} \vert \theta(r) id_\Omega &\vert_{C([T_1,T_2], C^m(\mathbb{T}^2))} \big) = \mathbb{E} \big( \vert id_\Omega \vert_{C([T_1,T_2+1], C^m(\mathbb{T}^2))} \big) \\
			 &= \mathbb{E}^{\mathbb{P}} \big( \vert id_\Omega \vert_{C([T_1,T_2+1], C^m(\mathbb{T}^2))} \big) \leq \mathbb{E}^{\mathbb{P}} \big( K(T_1,T_2+1) \big) < \infty.
\end{split}
\label{Sec6-eq7}
\end{equation}
Taking (\ref{Sec6-eq6}) and (\ref{Sec6-eq7}) into account, Proposition 4.1.3 in \cite{Arnold98} (\emph{the dichotomy of linear growth for stationary processes}) with the measurable mapping $ \psi \mapsto \vert \psi \vert_{C([0,1], C^m(\mathbb{T}^2))}$ implies that 
$$ \limsup_{t \to \pm \infty} \frac{ \vert \theta(t) \psi \vert_{C([0,1], C^m(\mathbb{T}^2))}}{\vert t \vert}=0 $$
on a $( \theta(t))_{t \in \mathbb{R}}$-invariant set $\mathcal{F} \ni \Omega_0 \subset \Omega$ with $\mathbb{P}(\Omega_0)=1 $. Therefore, for any $\delta >0$, $\psi \in \Omega_0$ there is a constant $T(\delta, \psi)=T(\delta, m, \psi)>0$ such that 
$$ \vert \psi(t) \vert_{C^m (\mathbb{T}^2)} = \vert \theta(t) \psi(0) \vert_{C^m (\mathbb{T}^2)} \leq  \vert \theta(t) \psi \vert_{C([0,1], C^m(\mathbb{T}^2))} \leq \delta \vert t \vert $$
for $ \vert t \vert \geq T(\delta, \psi)$. Hence, by (\ref{Sec6-eq4}) for any $\delta >0$, $\psi \in \Omega_0$ we have
\begin{equation}
\vert \psi(t) \vert_{C^m(\mathbb{T}^2)} = \vert \theta(t) \psi(0) \vert_{C^m(\mathbb{T}^2)} \leq \delta \vert t \vert + K(-T(\delta, \psi),T(\delta, \psi), \psi)
\label{Sec6-eq8}
\end{equation}
for all $ t \in \mathbb{R}$. Therefore, for $\tau>0 $ the mapping
$$ \psi \mapsto \begin{cases}
															 \int_{- \infty}^0 e^{\frac{s}{\tau}} \vert \psi(s) \vert_{C^m(\mathbb{T}^2)}^2 ds  & \text{ for } \psi \in \Omega_0 \\
																								0 & \text{ for } \psi \notin \Omega_0
								\end{cases} $$
is well--defined and $\mathcal{F}$--$\mathcal{B}(\mathbb{R})$ measurable, since for $\tau>0$ and $n \in \mathbb{N}$ the mappings
$$ \psi \mapsto \begin{cases}
															 \int_{- n}^0 e^{\frac{s}{\tau}} \vert \psi(s) \vert_{C^m(\mathbb{T}^2)}^2 ds  & \text{ for } \psi \in \Omega_0 \\
																								0 & \text{ for } \psi \notin \Omega_0
								\end{cases} $$
are finite, $\mathcal{F}$--$\mathcal{B}(\mathbb{R})$ measurable and the $\psi$--wise limits for $n \to \infty$ are finite by (\ref{Sec6-eq8}).
\end{proof}

\begin{prop}
Suppose Assumption \ref{Main-Assumption} holds and that there is $m \in \mathbb{N}$, $m \geq 2$, and $\gamma \in (0,1)$ such that
\begin{equation}
\sum_{k \in K} \lambda_k \alpha_k^{2 \gamma - 2H} \vert k \vert^{2m} < \infty \quad \text{and} \quad \sum_{k \in K} \lambda_k \alpha_k^{- 2H} \vert k \vert^{2m+2 \gamma} < \infty. 
\label{Sec6-eq9}
\end{equation}
Then the function $\phi: \mathbb{R} \times \Omega \times \mathbb{T}^2 \times \mathbb{R}^2 \longrightarrow \mathbb{T}^2 \times \mathbb{R}^2$
\begin{equation*}
(t, \psi, (x,y))	\mapsto \phi(t, \psi, (x,y)):= \phi(t, \psi) \binom{x}{y} := \binom{x(t)}{\dot{x}(t)}, 
\end{equation*}
defines a $C^{m-1}$--RDS over the ergodic metric dynamical system $(\Omega, \mathcal{F}, \mathbb{P}, (\theta(t))_{t \in \mathbb{R}})$ introduced in Proposition \ref{Sec6-Prop1}, where $\binom{x(t)}{\dot{x}(t)}  \in \mathbb{T}^2 \times \mathbb{R}^2$ is the unique global $C^{m-1}$--solution for $\tau>0$, $\psi \in \Omega$ and $(x,y) \in \mathbb{T}^2 \times \mathbb{R}^2$ at time $t \in \mathbb{R}$ to 
\begin{equation}
\frac{d}{dt} \binom{x(t)}{\dot{x}(t)} = f_{\psi, \tau}(t, (x(t), \dot{x}(t))), \ \binom{x(0)}{\dot{x}(0)}= \binom{x}{y},
\label{Sec6-eq10}
\end{equation}
with $f_{\psi, \tau}: \mathbb{R} \times \mathbb{T}^2 \times \mathbb{R}^2 \rightarrow \mathbb{R}^4$ defined by
\begin{equation*}
(t, x, y ) \mapsto f_{\psi, \tau}(t, (x, y)) = \begin{pmatrix}
	y \\
	\frac{1}{\tau} (\nabla^{\bot} \psi(x,t) -y)
\end{pmatrix}. 
\end{equation*}
Here we change $\Omega$ to $\Omega:= \Omega_0$, i.e. to the $(\theta(t))_{t \in \mathbb{R}}$--invariant set $\Omega_0$ introduced in Lemma \ref{Sec6-Lem1}(ii).
\label{Sec6-Prop2}
\end{prop}
\begin{proof}
First notice that by (\ref{Sec6-eq9}) and Corollary \ref{Sec4-Cor2} for all $(x,y) \in \mathbb{T}^2 \times \mathbb{R}^2$, $\psi \in \Omega$ (i.e. $\psi \in \Omega_0 $) there is a unique global $C^{m-1}$--solution $\binom{x(t)}{\dot{x}(t)}$ to equation (\ref{Sec6-eq10}). \\
Since $t \mapsto \phi(t, \psi)\binom{x}{y}$ is continuous for every $(\psi, (x,y)) \in \Omega \times \mathbb{T}^2 \times \mathbb{R}^2 $ and $(x,y) \mapsto \phi(t, \psi)\binom{x}{y}$ is continuous for every $(t,\psi) \in  \mathbb{R} \times \Omega $, to prove the measurability of $\phi$ by Lemma 1.1 in \cite{Crauel02} we only need to prove the measurability of $\psi \mapsto \phi(t, \psi)\binom{x}{y}$ for every $(t, (x,y)) \in \mathbb{R} \times \mathbb{T}^2 \times \mathbb{R}^2$. But this measurability is obvious: Since $\Omega$ is equipped with the compact open topology, $\psi \mapsto \psi(x,t)$ is measurable for every $(x, t) \in \mathbb{T}^2 \times \mathbb{R}$ and therefore also $\psi \mapsto \phi(t, \psi)\binom{x}{y}$. \\
The cocycle property is just a consequence of the uniqueness of the solution to equation (\ref{Sec6-eq10}). 
\end{proof}
The next theorem ensures the existence of the random $\mathcal{D}$--attractor.
\begin{thm}
Suppose all assumptions of Proposition \ref{Sec6-Prop2} hold. Then the $C^{m-1}$--RDS $\phi$ defined in Proposition \ref{Sec6-Prop2} has a unique random $\mathcal{D}$--attractor $\mathcal{A}$, where the universe of closed random sets $\mathcal{D}$ is given by
\begin{equation}
\begin{split}
\mathcal{D}:= \big\{ D \ | \ D &\text{ is a closed random set of } \mathbb{T}^2 \times \mathbb{R}^2 \text{ with } r_D(\psi):= \sup_{(x,y) \in D(\psi)} \vert (x,y) \vert < \infty \\
			 &\text{ and } r_D(\theta(-t) \psi) e^{-ct} \to 0 \text{ for } t \to \infty \text{ and any }  c>0, \ \psi \in \Omega      \big\}.
\end{split}
\label{Sec6-eq11}
\end{equation}
Further, for any $\delta>0$
\begin{equation}
B^\delta(\psi):= \left\{ (x,y) \in \mathbb{T}^2 \times \mathbb{R}^2  \big| \ \vert y \vert^2 \leq \frac{(1+ \delta)}{\tau} \int_{- \infty}^0 e^{\frac{u}{\tau}} \vert \psi(u) \vert^2_{C^1(\mathbb{T}^2)} du \right\}, \ \psi \in \Omega,
\label{Sec6-eq12}
\end{equation}
is a $\mathcal{D}$--absorbing and $\phi$--forward invariant closed random set. In particular, for any $\delta >0$ we have
$$ \mathcal{A}(\psi)= \bigcap_{t \in \mathbb{N}} \phi(t, \theta(-t) \psi)B^\delta(\theta(-t)\psi), \ \psi \in \Omega. $$
\label{Sec6-Thm2}
\end{thm}
\begin{proof}
In view of Theorem \ref{Sec6-Thm1} we only have to show that for any $\delta > 0$ the random set $B(\psi):= B^\delta(\psi)$ defined in (\ref{Sec6-eq12}) is a $\mathcal{D}$--absorbing and $\phi$--forward invariant closed random set. \\
Notice here that $ \psi \mapsto \int_{- \infty}^0 e^{\frac{u}{\tau}} \vert \psi(u) \vert^2_{C^1(\mathbb{T}^2)} du$ is measurable and finite for every $\psi \in \Omega$. This follows by Lemma \ref{Sec6-Lem1}(ii) and recall here again the change of $\Omega$ in Proposition \ref{Sec6-Prop2}. Further, the random set valued map $B$ takes values in closed bounded subsets of $\mathbb{T}^2 \times \mathbb{R}^2$ and is measurable since the random radius $\psi \mapsto r_C(\psi) = \sup_{(x,y) \in C(\psi)} \vert (x,y) \vert$ is measurable, where $C$ is a bounded closed random set. Therefore, $B$ is a closed random set. \\
In the following, for any set $C \subseteq \mathbb{T}^2 \times \mathbb{R}^2 $ we set
$$ \pi_{\mathbb{T}^2}(C):= \{ \pi_{\mathbb{T}^2}(x,y) | (x,y) \in C \}, \quad \pi_{\mathbb{R}^2}(C):= \{ \pi_{\mathbb{R}^2}(x,y) | (x,y) \in C \}, $$
where
$$ \pi_{\mathbb{T}^2}: \mathbb{T}^2 \times \mathbb{R}^2 \rightarrow \mathbb{T}^2, \ (x,y) \mapsto x, \quad \text{and} \quad \pi_{\mathbb{R}^2}: \mathbb{T}^2 \times \mathbb{R}^2 \rightarrow \mathbb{R}^2, \ (x,y) \mapsto y.$$
Next we prove that $\psi \mapsto B(\psi)$ is $\phi$--forward invariant, i.e. $ \phi(t, \psi) B(\psi) \subseteq B(\theta(t) \psi) $ for all $t \geq 0$, $\psi \in \Omega$. So to prove the $\phi$--forward invariance of $B$ we only have to show that
$$ \sup_{(x,y) \in B(\psi)} \big\vert \pi_{\mathbb{R}^2}(\phi(t, \psi) \binom{x}{y})  \big\vert^2 = \sup_{(x,y) \in B(\psi)} \{ \vert \dot{x}(t)  \vert^2 \ | \ \dot{x}(0)=y \} \leq  \sup_{(x,y) \in B(\theta(t) \psi)} \vert y \vert^2 $$
for any $t \geq 0$, $\psi \in \Omega$, since $\pi_{\mathbb{T}^2}(B(\theta(t) \psi)) = \mathbb{T}^2$ for any $t \geq 0$ and $\psi \in \Omega$. \\
To estimate $\vert \dot{x}(t) \vert^2$ for $t>0$, we take the inner product of the equation
$$ \tau \ddot{x}(s) = \nabla^{\bot} \psi ( x(s),s ) - \dot{x}(s) $$
with $\dot{x}(s)$ and obtain
\begin{equation*}
\begin{split}
\tau \frac{1}{2} \frac{d}{ds} \vert \dot{x}(s) \vert^2 = \langle \dot{x}(s), \nabla^{\bot} \psi ( x(s),s ) \rangle - \vert \dot{x}(s) \vert^2 &\leq \frac{1}{2} \vert \nabla^{\bot} \psi ( x(s),s ) \vert^2 + \frac{1}{2} \vert \dot{x}(s) \vert^2 - \vert \dot{x}(s) \vert^2 \\
				&= \frac{1}{2} \vert \nabla^{\bot} \psi ( x(s),s ) \vert^2 - \frac{1}{2} \vert \dot{x}(s) \vert^2,
\end{split}
\end{equation*}
where we used $\langle z_1, z_2 \rangle \leq \frac{1}{2} \vert z_1 \vert^2 + \frac{1}{2} \vert z_2 \vert^2$ for $z_1,z_2 \in \mathbb{R}^2$. \\
Multiplying by $\frac{2}{\tau} e^{\frac{s}{ \tau}}$ on each side gives
\begin{equation}
\frac{d}{ds} \big( e^{\frac{s}{\tau}} \vert \dot{x}(s) \vert^2  \big) \leq \frac{1}{\tau} e^{\frac{s}{\tau}} \vert \nabla^{\bot} \psi ( x(s),s ) \vert^2 .
\label{Sec6-eq13}
\end{equation}
By integrating the inequality (\ref{Sec6-eq13}) from $0$ to $t$ we get
$$ e^{\frac{t}{\tau}} \vert \dot{x}(t) \vert^2 - \vert \dot{x}(0) \vert^2 \leq \frac{1}{\tau} \int_0^t e^{\frac{s}{\tau}} \vert \nabla^{\bot} \psi ( x(s),s ) \vert^2 ds $$
so that
\begin{equation*}
\begin{split}
\vert \dot{x}(t) \vert^2 &\leq e^{-\frac{t}{\tau}} \vert \dot{x}(0) \vert^2 + \frac{1}{\tau} \int_0^t e^{-\frac{(t-s)}{\tau}} \vert \nabla^{\bot} \psi ( x(s),s ) \vert^2 ds \\
			&\leq e^{-\frac{t}{\tau}} \vert \dot{x}(0) \vert^2 + \frac{1}{\tau} \int_0^t e^{-\frac{(t-s)}{\tau}} \vert \psi(s) \vert^2_{C^1(\mathbb{T}^2)} ds \\
			&\stackrel{u = s-t}{=} e^{-\frac{t}{\tau}} \vert \dot{x}(0) \vert^2 + \frac{1}{\tau} \int_{-t}^0 e^{\frac{u}{\tau}}  \vert \psi(u+t) \vert^2_{C^1(\mathbb{T}^2)} du.
\end{split}
\end{equation*}
Therefore
\begin{equation}
\big\vert \pi_{\mathbb{R}^2}(\phi(t, \psi) \binom{x(0)}{\dot{x}(0)})  \big\vert^2 = \vert \dot{x}(t) \vert^2 \leq e^{-\frac{t}{\tau}} \vert \dot{x}(0) \vert^2 + \frac{1}{\tau} \int_{-t}^0 e^{\frac{u}{\tau}}  \vert \psi(u+t) \vert^2_{C^1(\mathbb{T}^2)} du.
\label{Sec6-eq14}
\end{equation}
Now by (\ref{Sec6-eq14}) and the definition of $B(\psi)$ we obtain for $t \geq 0$
\begin{equation*}
\begin{split}
\sup_{(x,y) \in B(\psi)} \big\vert &\pi_{\mathbb{R}^2}(\phi(t, \psi) \binom{x}{y})  \big\vert^2 \leq e^{-\frac{t}{\tau}}  \sup_{(x,y) \in B(\psi)} \vert y \vert^2 + \frac{1}{\tau} \int_{-t}^0 e^{\frac{u}{\tau}}  \vert \psi(u+t) \vert^2_{C^1(\mathbb{T}^2)} du \\
			 &\leq \frac{(1+ \delta)}{\tau} e^{-\frac{t}{\tau}} \int_{- \infty}^0 e^{\frac{u}{\tau}} \vert \psi(u) \vert^2_{C^1(\mathbb{T}^2)} du + \frac{(1+ \delta)}{\tau} \int_{-t}^0 e^{\frac{u}{\tau}}  \vert \psi(u+t) \vert^2_{C^1(\mathbb{T}^2)} du \\
			&= \frac{(1+ \delta)}{\tau} \int_{- \infty}^{-t} e^{\frac{u}{\tau}} \vert \psi(u+t) \vert^2_{C^1(\mathbb{T}^2)} du + \frac{(1+ \delta)}{\tau} \int_{-t}^0 e^{\frac{u}{\tau}}  \vert \psi(u+t) \vert^2_{C^1(\mathbb{T}^2)} du \\
			&= \frac{(1+ \delta)}{\tau} \int_{- \infty}^{0} e^{\frac{u}{\tau}} \vert \psi(u+t) \vert^2_{C^1(\mathbb{T}^2)} du \\
			&=  \sup_{(x,y) \in B(\theta(t) \psi)} \vert y \vert^2.
\end{split}
\end{equation*}
So $\psi \mapsto B(\psi)$ is $\phi$--forward invariant. \\
Finally, we prove that $\psi \mapsto B(\psi)$ is $\mathcal{D}$--absorbing. For any $D \in \mathcal{D}$ where $\mathcal{D}$ is defined in (\ref{Sec6-eq11}) we have 
$$ \pi_{\mathbb{T}^2}(\phi(t,\theta(-t) \psi)D((\theta(-t))\psi)) \subseteq \pi_{\mathbb{T}^2}(B(\psi))= \mathbb{T}^2 $$
and by (\ref{Sec6-eq14})
\begin{equation}
\begin{split}
\sup_{(x,y) \in D(\theta(-t)\psi)} \big\vert \pi_{\mathbb{R}^2} &(\phi(t,\theta(-t) \psi) \binom{x}{y})  \big\vert^2 \\
		&\leq e^{-\frac{t}{\tau}} \sup_{(x,y) \in D(\theta(-t)\psi)} \vert y \vert^2 + \frac{1}{\tau} \int_{-t}^0 e^{\frac{u}{\tau}}  \vert \psi(u+t-t) \vert^2_{C^1(\mathbb{T}^2)} du \\
		 &= e^{-\frac{t}{\tau}} \sup_{(x,y) \in D(\theta(-t)\psi)} \vert y \vert^2 + \frac{1}{\tau} \int_{-t}^0 e^{\frac{u}{\tau}}  \vert \psi(u) \vert^2_{C^1(\mathbb{T}^2)} du
\end{split}
\label{Sec6-eq15}
\end{equation}
for any $t \geq 0$ and $\psi \in \Omega$. By the definition of the set $\mathcal{D}$ the first term on the right hand side of (\ref{Sec6-eq15}) converges for $t \to \infty$ to zero. The second term tends for $t \to \infty$ to $\frac{1}{\tau} \int_{- \infty}^{0} e^{\frac{u}{\tau}} \vert \psi(u) \vert^2_{C^1(\mathbb{T}^2)} du$. Further, Lemma \ref{Sec6-Lem1}(ii) implies that 
$$ e^{-tc} \frac{(1+ \delta)}{\tau} \int_{- \infty}^{0} e^{\frac{u}{\tau}} \vert \psi(u) \vert^2_{C^1(\mathbb{T}^2)} du \rightarrow 0 \quad \text{as } t \to \infty$$
for any $c, \delta >0$ and $\psi \in \Omega$. This ensures $B \in \mathcal{D}$, i.e. $B$ is also $\mathcal{D}$--absorbing. The assertion follows now by Theorem \ref{Sec6-Thm1}.
\end{proof}
\begin{rem}
The universe of closed random sets $\mathcal{D}$ defined in Theorem \ref{Sec6-Thm2} contains only sets which do not grow with exponential speed (or grow \emph{sub--exponentially fast}). Notice also that $ \lim_{t \to \infty} e^{-ct} r_D(\theta(-t) \psi)=0 $ for any $c>0$, $\psi \in \Omega$ where $D$ is a closed random set is equivalent to $ \lim_{t \to \infty} \log( \max \{1, r_D(\theta(-t) \psi) \})/t=0 $ for any $\psi \in \Omega$. Such subsets are also called \emph{tempered}.
\end{rem}
The next step would be to derive (upper) bounds of the \emph{Hausdorff dimension} of the random $\mathcal{D}$--attractor which is due to ergodicity of the underlying metric dynamical system $ \left(\Omega,\mathcal{F}, \mathbb{P}, \theta(t))_{t \in \mathbb{R}} \right)$ $\mathbb{P}$--a.s. constant. This will be studied also in the fractional noise case in a forthcoming paper.


\section{Simulation}

In this section we visualize the long--time behaviour of particle motions. We simulate the motion of $10^4$ particles uniformly distributed on the unit square at time zero with zero velocities for $5000$ time units according to the system ($M$) with $H=1/3$, $\nu = 10^{-2}$ and the Kolmogorov spectrum, but with different values of $\tau$. As already described in Section 5, we set the spectrum $(\lambda_k)_{k \in K}$ of $Q$ for $\vert k \vert > 2 \pi R$ with some fixed $R \in \mathbb{N}$ to zero, actually $R=2$ in our simulation. Then 
\begin{equation*}
v(x,t) = \sum_{k \in K, \ \vert k \vert \leq 2 \pi R} i \binom{k_2}{-k_1} \hat{\psi}_k(t) e^{i<k,x>}, \quad x \in \mathbb{T}^2, t \in \mathbb{R},
\end{equation*}
where $\hat{\psi}_k (t) = \sqrt{\lambda_k} \nu^H \int_{-\infty}^t e^{-(t-u) \nu \alpha_k} d\beta_k^H(u)$. To simulate $\hat{\psi}_k$, we use the \emph{Cholesky method} (see e.g. \cite{Asmussen07} pp. 311), since we know by Proposition \ref{Section3-Prop1}(i) and Remark \ref{Sec4-Rem1} the covariance structure of $\hat{\psi}_k$. Given a realisation of $\hat{\psi}_k$, $k \in K, \ \vert k \vert \leq 2 \pi R$, the velocity field $v$ can be computed efficiently using the \emph{fast Fourier transform algorithm}, see Chapter 4 in \cite{Sigurgeirsson01}. Finally, we integrate the trajectory of the inertial particle using the classic fourth--order Runge--Kutta scheme with time step double of that used for the velocity field. \\
Figure \ref{Figure1} shows the final position of the particles in the phase space associated to four different Stokes' number in form of $\tau$. The clustering is distinctive for $\tau = 1$ and $\tau = 10^{-1}$, but there is almost no clustering for very low and large values of $\tau$, i.e. for $\tau = 10^{-4}$ and $\tau = 10^2$ in the experiment. Repeating the simulation under same conditions but with different Hurst parameter produces similar results. Therefore, we conclude that also the generalized model with fractional noise captures the clustering phenomenon of preferential concentration. We leave the numerical study of the dependence of the clustering on the parameters $H$ and $\nu$ for a future publication.
\begin{figure}
\begin{center}
\includegraphics{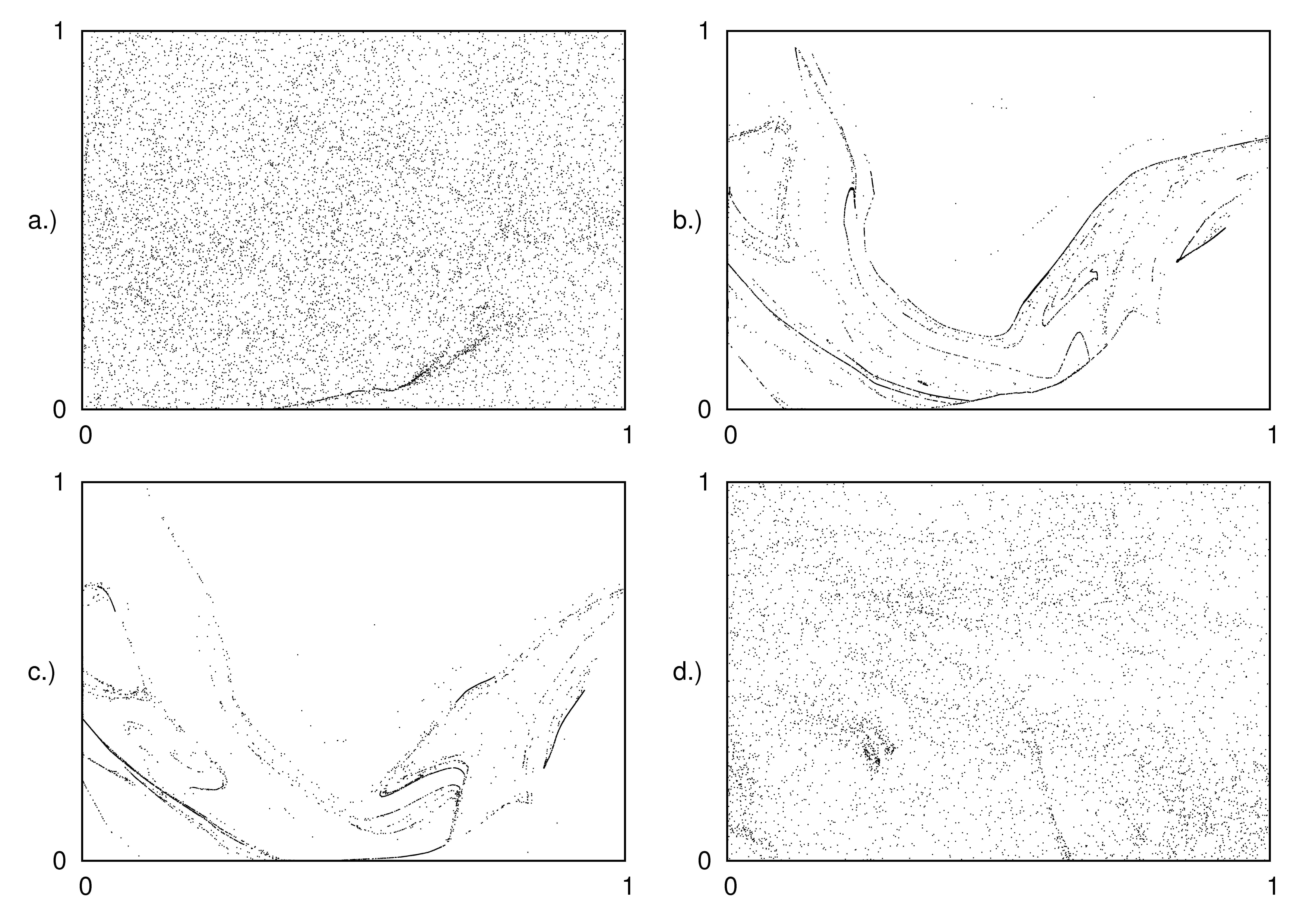}
\caption{Snapshots of the final positions of $10^4$ particles in the phase space associated to different Stokes' numbers $\tau$. a.) $\tau = 10^{-4}$, b.) $\tau = 10^{-1}$, c.) $\tau =  1$, d.) $\tau= 10^2$.}
\label{Figure1}
\end{center}
\end{figure}


\section*{Acknowledgements}
The author is very grateful to Wilhelm Stannat, Andrew Stuart, Dirk Bl\"omker and Andrew Duncan for helpful advice and suggestions. 


\end{document}